\documentclass[11pt, oneside]{amsart}
\usepackage[text={5.58in,8.5in},centering,letterpaper,dvips]{geometry}
\usepackage[dvipsnames]{xcolor}
\usepackage{graphicx}
\usepackage{subcaption}
\usepackage{amsfonts}
\usepackage{epsf}
\usepackage{amssymb}
\usepackage{amsmath}
\usepackage{amscd}
\usepackage{tikz}
\usepackage{pdfpages}
\usepackage{fancyhdr}
\usepackage{setspace}
\usepackage{hyperref}
\usepackage[all]{xy}
\usetikzlibrary{matrix}
\usepackage{verbatim}
\usepackage{enumerate}

\theoremstyle{theorem}
\newtheorem{theorem}{Theorem}[section]
\newtheorem{proposition}[theorem]{Proposition}
\newtheorem{lemma}[theorem]{Lemma}
\newtheorem{question}[theorem]{Question}
\newtheorem{corollary}[theorem]{Corollary}
\newtheorem{conjecture}[theorem]{Conjecture}

\theoremstyle{definition}

\newtheorem{remark}[theorem]{Remark}

\newtheorem*{heuristic}{Heuristic}

\newcommand{\F}{\mathcal{F}}
\newcommand{\Z}{\mathbb{Z}}

\newcommand{\R}{\mathfrak{R}}
\newcommand{\D}{\mathcal{D}}

\newcommand{\A}{\alpha}
\newcommand{\n}{\beta}
\newcommand{\X}{\times}
\newcommand{\pd}{\partial}
\newcommand{\eps}{\epsilon}
\newcommand{\Ss}{\mathcal{S}}

\newcommand{\K}{\mathcal{K}}
\newcommand{\sgn}{\text{sgn}}
\newcommand{\ord}{\text{ord}}

\makeatletter
\def\@seccntformat#1{%
  \protect\textup{\protect\@secnumfont
    \ifnum\pdfstrcmp{subsection}{#1}=0 \bfseries\fi
    \csname the#1\endcsname
    \protect\@secnumpunct
  }%
}  
\makeatother

\makeatletter
\newtheorem*{rep@theorem}{\rep@title}
\newcommand{\newreptheorem}[2]{%
\newenvironment{rep#1}[1]{%
 \def\rep@title{#2 \ref{##1}}%
 \begin{rep@theorem}}%
 {\end{rep@theorem}}}
\makeatother

\newreptheorem{theorem}{Theorem}
\newreptheorem{lemma}{Lemma}
\newreptheorem{question}{Question}
\newreptheorem{corollary}{Corollary}
\newreptheorem{proposition}{Proposition}

\begin{document}

\rhead{\thepage}
\lhead{\author}
\thispagestyle{empty}


\raggedbottom
\pagenumbering{arabic}
\setcounter{section}{0}


\title{Bounding the ribbon numbers of knots and links}

\author{Stefan Friedl}

\author{Filip Misev}

\author{Alexander Zupan}

\begin{abstract}
The ribbon number $r(K)$ of a ribbon knot $K \subset S^3$ is the minimal number of ribbon intersections contained in any ribbon disk bounded by $K$.  We find new lower bounds for $r(K)$ using $\det(K)$ and $\Delta_K(t)$, and we prove that the set $\mathfrak{R}_r~=~\{\Delta_K(t)~:~r(K)~\leq~r\}$ is finite and computable.  We determine $\mathfrak{R}_2$ and $\mathfrak{R}_3$, applying our results to compute the ribbon numbers for all ribbon knots with 11 or fewer crossings, with three exceptions.  Finally, we find lower bounds for ribbon numbers of links derived from their Jones polynomials.
\end{abstract}

\maketitle

\section{Introduction}

At the inception of knot theory, crossing number emerged as the original knot invariant, the minimal number of crossings in any diagram $D$ for $K$.  Crossing numbers have been used for over a century to organize knot tables by complexity, guided by the principle that for a given $n$, it is possible to enumerate all of the finitely many possible knot diagrams with exactly $n$ crossings (distinguishing them is a more subtle task).  In this paper, we seek to understand another invariant, the \emph{ribbon number} of a ribbon knot $K \subset S^3$.  We say that $K$ is \emph{ribbon} if $K$ is the boundary of an immersed disk $\D$ in $S^3$ with only ribbon singularities (see Figure~\ref{fig:challenge}), called a \emph{ribbon disk}, and the \emph{ribbon number} $r(K)$ of a ribbon knot $K$ is the minimum of $r(\D)$ taken over all ribbon disks $\D$ bounded by $K$.  Note that formally, a knot diagram $D$ is an immersed curve in 2-space, but we can perturb $D$ near its crossings to realize the corresponding knot $K$ embedded in 3-space.  Similarly, a ribbon disk $\D$ is immersed in 3-space, but we can perturb $\D$ near its ribbon intersections to construct a smoothly embedded disk in $D^4$.  This relationship, and the fact that both invariants involve minimizing self-intersections, gives rise to a heuristic:


\begin{heuristic}
Ribbon number is like a crossing number for ribbon disks.
\end{heuristic}

In trying to compute ribbon numbers, however, one encounters an issue that threatens the heuristic:  For any $r \geq 2$, there are infinitely many different knots $K$ such that $r(K) =r$, and so this presents a difficulty.  To address this issue, we use Alexander polynomials, proving that the set $\R_r$, the set of all Alexander polynomials $\Delta_K(t)$ of ribbon knots $K$ such that $r(K) \leq r$, is a finite set.

\begin{theorem}\label{thm:main2}
For each $r$, the set $\R_r = \{\Delta_K(t) : r(K) \leq r\}$ is finite and computable.  In particular,
\begin{eqnarray*}
\R_2 &=& \{1, \Delta_{3_1 \# \overline{3_1}}, \Delta_{6_1}\}; \\
\R_3 &=&  \{1, \Delta_{3_1 \# \overline{3_1}}, \Delta_{6_1}, \Delta_{8_8}, \Delta_{8_9}, \Delta_{9_{27}}, \Delta_{9_{41}}, \Delta_{10_{137}}, \Delta_{10_{153}}, \Delta_{11n_{116}}\}.
\end{eqnarray*}
\end{theorem}

This theorem is sufficiently powerful to allow us to determine the ribbon numbers for all ribbon knots in the knot table up to 11 crossings, with three exceptions: $10_{123}$, $11a_{164}$, and $11a_{326}$, whose (potentially minimal) ribbon disks are shown in Figure~\ref{fig:challenge}.

\begin{theorem}
The ribbon numbers for knots up to 11 crossings, except for $10_{123}$, $11a_{164}$, and $11a_{326}$, are computed in Table~\ref{table1} (knots up to 10 crossings) and Table~\ref{table2} (knots with 11 crossings).
\end{theorem}

\begin{figure}[h!]
  \centering
  \includegraphics[width=.3\linewidth]{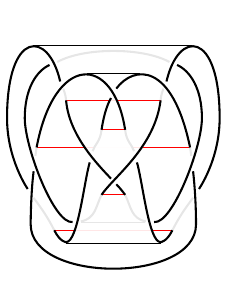} \qquad
    \includegraphics[width=.2\linewidth]{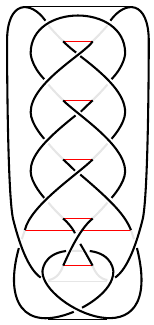} \qquad
      \includegraphics[width=.2\linewidth]{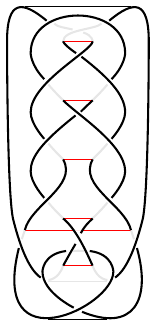}
  \caption{Ribbon disks for $10_{123}$ (left), $11a_{164}$ (center), and $11a_{326}$ (right).  Do the disks realize the ribbon numbers of these knots?}
  \label{fig:challenge}
\end{figure}

The proof of Theorem~\ref{thm:main2} involves associating to any ribbon disk $\D$ a combinatorial object called a \emph{ribbon code}, a tree with markings on its edges, that can be used to compute $\Delta_{\pd \D}(t)$.  We then obtain the explicit descriptions of $\R_2$ and $\R_3$ by enumerating all possible ribbon codes.

Along the way to finding $\R_2$ and $\R_3$, we develop more general bounds for all ribbon numbers involving Alexander polynomials and knot determinants.  It is well-known that if $K$ is ribbon, then its Alexander polynomial $\Delta_K(t)$ can be expressed as $\Delta_K(t) = f(t) \cdot f(t^{-1})$ for some $f(t) \in \Z[t,t^{-1}]$ such that $f(1) = \pm 1$.  Using a result of Yasuda~\cite{YasudaAlex}, we prove

\begin{theorem}\label{thm:main1}
Suppose $K$ is a ribbon knot such that $r(K) = r$.  Then $\Delta_K(t)$ can be expressed as $f(t) \cdot f(t^{-1})$ for $f(t) \in \Z[t]$ such that
\[ f(t) = a_0 + a_1t + \dots + a_rt^r \qquad \text{and} \qquad |a_i| \leq \binom{r}{i}.\]
\end{theorem}

As a corollary, we obtain

\begin{corollary}\label{cor:main1}
Suppose $K$ is a ribbon knot.  Then
\[ \det K \leq (2^{r(K)}-1)^2.\]
\end{corollary}

\begin{remark}
The curious reader might wonder how the bounds given by Theorem~\ref{thm:main2}, Theorem~\ref{thm:main1}, and Corollary~\ref{cor:main1} are related.  For $r \leq 3$, Theorem~\ref{thm:main2} is strictly stronger than Theorem~\ref{thm:main1}, and Theorem~\ref{thm:main1} is strictly stronger than Corollary~\ref{cor:main1}:  For example, if $K$ is the knot $11n_{49}$ or $11n_{116}$,  Theorem~\ref{thm:main1} implies that $r(K) \geq 2$, while $\Delta_{K}(t) \notin \R_2$, and so Theorem~\ref{thm:main2} implies $r(K) \geq 3$.  For $K' = 10_3$, Theorem~\ref{thm:main1} yields $r(K') \geq 3$, while $\Delta_{K'}(t) \notin \R_3$, and so Theorem~\ref{thm:main2} says $r(K') \geq 4$.  For $K'' = 10_{153}$, we have $\det(K'') = 1$ and so Corollary~\ref{cor:main1} provides no useful information.  Theorem~\ref{thm:main1}, however, asserts that $r(K'') \geq 3$.
\end{remark}

Finally, we prove an analogue of Corollary~\ref{cor:main1} for ribbon links, where a ribbon link $L$ bounds a collection of immersed disks in $S^3$ with only ribbon singularities.  In~\cite{eisermann}, Eisermann proved that if $L$ is an $n$-component ribbon link, then its Jones polynomial $V_L(q)$ is divisible by the Jones polynomial of the $n$-component unlink, $(q+q^{-1})^{n-1}$, and so the \emph{generalized Jones determinant} $\det_n(V_L)$ can be defined as
\[ \text{det}_n(V_L) = \left(\frac{V_L(q)}{(q+q^{-1})^{n-1}}\right)_{q=i},\]
where the classical determinant $\det(L)$ agrees with $|\det_1(V_L)|$.  We prove

\begin{theorem}\label{thm:main3}
Suppose $L$ is an $n$-component ribbon link.  Then
\[ |\text{det}_n(V_L)| \leq 9^{r(L)}.\]
\end{theorem}

\begin{remark}
The invariant $r(K)$ has seen relatively little attention in the knot theory literature.  To the best of our knowledge, it first appeared in~\cite{mizuma}, in which Mizuma proved that that for $K$ the Kinoshita-Terasaka knot $11n42$, $r(K) = 3$.  In~\cite{aceto}, Aceto examined the relationship between ribbon numbers and \emph{symmetric ribbon numbers}.  A related notion is the \emph{ribbon crossing number} of a ribbon 2-knot, which has been more thoroughly examined (see, for instance,~\cite{KanTak,YasudaBase,YasudaAlex}).  The reader should be aware that although there are earlier references to knots of ``ribbon number one" (see, for example,~\cite{bleiler-eudave} and~\cite{tanaka}), these instances refer to an invariant defined as the minimum number of bands in a disk-band presentation for a ribbon disk $\D$ bounded by $K$, now more commonly referred to as the \emph{fusion number} $\F(K)$ of $K$.  We define a disk-band presentation below in Section~\ref{sec:prelim}.
\end{remark}

\begin{remark}
A forthcoming manuscript~\cite{polymath} will compute the set $\R_4$ and will use this information to extend the work in this paper to the collection of 12-crossing ribbon knots.
\end{remark}

\subsection{Organization}

In Section~\ref{sec:prelim}, we introduce some elementary bounds on ribbon numbers, and as a proof of concept, we use these bounds to find the ribbon numbers of the knots $T_{p,q} \# \overline{T_{p,q}}$.  In Section~\ref{sec:doubles}, we relate ribbon disks to ribbon 2-knots, prove a folk theorem (Lemma~\ref{lem:double}) about the Alexander polynomial of a ribbon knot, and establish Theorem~\ref{thm:main1} and Corollary~\ref{cor:main1}.  In Section~\ref{sec:codes}, we introduce ribbon codes and explain a procedure by which a ribbon code determines the Alexander polynomial for the corresponding knot (Proposition~\ref{prop:block} and Corollary~\ref{cor:codes}).  In Section~\ref{sec:code}, we discuss simplification of ribbon codes and prove Propositions~\ref{prop:r2} and~\ref{prop:r3}, enumerating all possible Alexander polynomials of ribbon knots with ribbon number at most 2 and 3, respectively.  In Section~\ref{sec:tab}, we combine our results to tabulate ribbon numbers for ribbon knots up to 11 crossings, with data shown in Tables~\ref{table1} and~\ref{table2}.  In Section~\ref{sec:jones}, we pivot to ribbon links and prove Theorem~\ref{thm:main3}, a generalization of Corollary~\ref{cor:main1}.  Finally, in Section~\ref{sec:conj}, we state several conjectures and questions to motivate future work.

\subsection{Acknowledgements}

Part of this work was completed while the third author was a guest at MPIM, and he is grateful for the institute's support.  The third author also appreciates the hospitality of the first two authors during visits to the University of Regensburg, thanks Jeffrey Meier for helpful conversations, and acknowledges his Polymath Jr. REU group from the summer of 2023 for their energy and insights.  Finally, we thank the authors of~\cite{KSTI} for email exchanges related to their work.  The first and second author were supported by the CRC 1085 ``higher invariants'' at the University of Regensburg.  The third author was supported by NSF awards DMS-2005518 and DMS-2405301 and a Simons Fellowship.

\section{Preliminaries}\label{sec:prelim}

We work in the smooth category.  In this section, we state several elementary bounds for ribbon numbers and use these bounds to determine the ribbon numbers $r(T_{p,q} \# \overline{T_{p,q}})$, where $\overline{K}$ denotes the mirror image of $K$ and $T_{p,q}$ is the $(p,q)$-torus knot. First, we explore upper bounds, which we obtain by explicit construction and which are related to symmetric union presentations of ribbon knots.  A knot diagram $D^*$ is a \emph{symmetric union presentation} if $D^*$ has a vertical axis of symmetry $L$ such that
\begin{enumerate}
\item Outside a small neighborhood of $L$, the diagram $D^*$ has reflection symmetry over $L$, and
\item $D^*$ meets $L$ in two horizontal strands and some number (possibly zero) of additional crossings.
\end{enumerate}
See the left panel of Figure~\ref{fig:symm} for an example.  A more precise definition appears in~\cite{lamm1}.  It is known that every knot with a symmetric union presentation is ribbon (examples of ribbon disks arising from symmetric union presentations are shown in Figure~\ref{fig:challenge}).  In addition, all ribbon knots with 10 or fewer crossings admit a symmetric union presentation, but it is open whether every ribbon knot admits such a presentation~\cite{lamm1,lamm2,lamm3}.

Given a symmetric union presentation $D^*$, we form the corresponding \emph{partial diagram} $D$ by cutting $D^*$ along the axis of symmetry $L$, vertically smoothing each crossing, and connecting the two horizontal strands by a vertical arc, as shown at right in Figure~\ref{fig:symm}.  It can be shown that since $D^*$ represents a knot, $D$ is connected.  In the special case that $D^*$ has no crossings on $L$, then $D^* = D \# \overline{D}$, and so we see that a symmetric union presentation for a ribbon knot generalizes the well-known fact that for any knot $K$, the connected sum $K \# \overline{K}$ is ribbon.

\begin{figure}[h!]
  \centering
  \includegraphics[width=.6\linewidth]{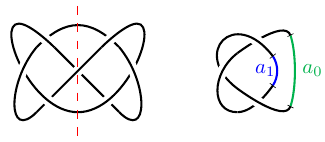}
  \caption{At left, an example of a symmetric union presentation $D^*$ for $6_1$.  At right, the corresponding partial diagram $D$, with vertical arcs labeled as in the proof of Lemma~\ref{lem:underpass}.}
\label{fig:symm}
\end{figure}

\begin{lemma}\label{lem:underpass}
Suppose $D^*$ is a symmetric union presentation for a ribbon knot $K$, where $D^*$ has partial diagram $D$, and a horizontal strand of $D$ is adjacent to $\ell$ consecutive under-crossings away from the axis of symmetry $L$.  Then
\[ r(K) \leq c(D) - \ell.\]
\end{lemma}

\begin{proof}
Consider $D$ as a diagram with underlying surface $S^2$, let $a_0$ denote the vertical arc of $D$ obtained by connecting horizontal strands of $D^*$, and let $a_1,\dots,a_k$ denote the arcs in $D$ obtained by smoothing the crossings of $D^*$ along the axis of symmetry $L$.  Remove a small disk neighborhood of $a_0$ to get a diagram $D'$ for a knotted arc $\A$ with underlying surface $D^2$.  By hypothesis, one endpoint of $\A^+$ is adjacent to the $\ell$ consecutive under-crossings.  Let $N = D^2 \X [-1-\eps,1+\eps]$ be an embedded 3-ball in $S^3$, and use the diagram $D'$ to embed $\A^+$ and $\A^- = \overline{\A^+}$ in small collar neighborhoods of $D^2 \X \{1\}$ and $D^2 \X \{-1\}$, respectively.

Now, each point $x^+ \in \A^+$ is connected via a subinterval $I_x$ of a vertical fiber in $N$ to a point $x^- \in \A^-$.  Define
\[ \D = \bigcup_{x^+ \in \A^+} I_{x^+}.\]
Let $J$ be the knot associated to the diagram $D$.  Then $\D$ is a ribbon disk with $\pd \D = J \# \overline{J}$, and $\D$ contains a ribbon intersection corresponding to each crossing of $\A$, for a total of $c(\A) = c(D)$ ribbon intersections.  However, the $\ell$ ribbon intersections corresponding to consecutive under-crossings can be removed via an isotopy of $\D$, resulting in a new ribbon disk $\D'$ for $J \# \overline{J}$ such that $r(\D') = c(D) - \ell$.

Finally, for $1 \leq i \leq k$, let $a_i^{\pm}$ denote the arcs corresponding to $a_i$ in $\A^{\pm}$, and for $1 \leq i \leq k$, define
\[ R_i = \bigcup_{x^+ \in a_i^+} I_{x^+}.\]
Then $R_i$ is an embedded rectangle in $\D'$, and we can replace $R_i$ with a half-twisted rectangle $\widetilde{R_i}$ corresponding to the crossing in $D^*$ that was smoothed in the construction of the partial diagram $D$.  Replacing each rectangle $R_i$ with $\widetilde{R_i}$ yields a ribbon disk $\D^*$ such that $r(\D^*) = r(\D')$ and $\pd \D^* = K$, completing the proof.
\end{proof}

\begin{remark}
For a diagram $D$, one can define the \emph{maximal bridge length} $\ell(D)$ to be the maximum consecutive number of under-crossings (or over-crossings) contained in $D$.  Lemma~\ref{lem:underpass} then implies that if $D$ is a diagram for a knot $J$, we have
\[ r(J \# \overline{J}) \leq c(D) - \ell(D).\]
The quantity $c(D) - \ell(D)$ appears elsewhere as an upper bound for degrees of various polynomials.  See, for instance,~\cite{Kidwell},~\cite{KidStoim},~\cite{stoim2}, and~\cite{Thistlethwaite}.
\end{remark}

An example of carrying out this construction for the symmetric union presentation of $6_1$ from Figure~\ref{fig:symm} is shown in Figure~\ref{fig:underpass}. 

\begin{figure}[h!]
  \centering
  \includegraphics[width=.8\linewidth]{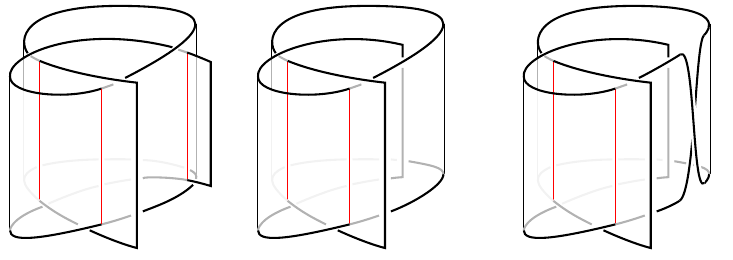}
  \caption{At left, construction of the ribbon disk $\D$ from the diagram $\A$ and removing a ribbon intersection via isotopy to get $\D'$. At right, a ribbon disk for $6_1$ obtained from the symmetric union presentation $D^*$ shown in Figure~\ref{fig:symm}.}
\label{fig:underpass}
\end{figure}

Now, we turn our attention to the more difficult task of finding lower bounds for ribbon numbers.  The first such bound involves the genus $g(K)$ and was observed in~\cite{mizuma}.

\begin{lemma}\label{lem:genus}
Let $K$ be a ribbon knot.  Then
\[ g(K) \leq r(K).\]
\end{lemma}

\begin{proof}
Suppose $K$ bounds a ribbon disk $\D$ with $r$ ribbon intersections. Each ribbon intersection can be smoothed as in Figure~\ref{fig:smooth} to produce an embedded, orientable surface $F$ such that $g(F) = r$ and $\pd F = \pd \D = K$.
\end{proof}

\begin{figure}[h!]
  \centering
  \includegraphics[width=.6\linewidth]{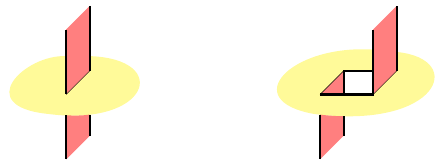}
  \caption{At left, a ribbon intersection in $\D$.  At right, the smoothing of the intersection to obtain $F$.}
  \label{fig:smooth}
\end{figure}

A similar bound can be obtained using the unknotting number $u(K)$.

\begin{lemma}
Let $K$ be a ribbon knot.  Then
\[ \frac{u(K)}{2} \leq r(K).\]
\end{lemma}

\begin{proof}
Suppose $K$ bounds a ribbon disk $\D$ with $r$ ribbon intersections.  Then there exists a diagram $D$ for $K$ with the property that each ribbon intersection can be removed with two crossing changes, as shown in Figure~\ref{fig:change}. The resulting knot $K'$ bounds an embedded disk~$\D'$, and it follows that $u(K) \leq 2r$.
\end{proof}

\begin{figure}[h!]
  \centering
  \includegraphics[width=.6\linewidth]{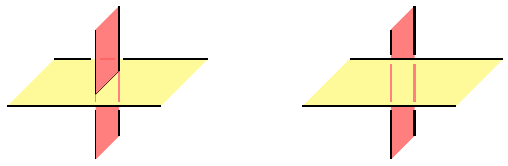}
  \caption{At left, a ribbon intersection in $\D$.  At right, two crossing changes remove the ribbon intersection.}
\label{fig:change}
\end{figure}


A \emph{disk-band presentation} for a ribbon disk $\D$ consists of a pair $(D,B)$, where $D$ is a collection of $n$ pairwise disjoint disks in $S^3$, and $B$ is a collection of embedded pairwise disjoint rectangles, the {\em bands}, such that each rectangle meets the interior of each disk transversely in a collection of arcs, meets $\pd (\bigcup D)$ in a pair of boundary arcs, and such that $\D = (\bigcup D) \cup (\bigcup B)$.  In this case, we note that $B$ contains $n-1$ bands.  Every ribbon disk $\D$ has a disk-band presentation, and the \emph{fusion number} $\F(K)$ is defined to be the minimum number of bands in a disk-band presentation for a ribbon disk $\D$ bounded by $K$.

\begin{lemma}\label{lem:fusion}
Let $K$ be a ribbon knot.  Then
\[ \F(K) \leq r(K)-1.\]
\end{lemma}

\begin{proof}
Let $\D$ be a ribbon disk for $K$ with $r$ ribbon intersections.  Each ribbon intersection of $\D$ gives rise to two arcs in $\D$, one of which, call it $a'_i$ has its endpoints on $\pd \D$, and the other of which, call it $a_i$, has its endpoints in $\text{int}(\D)$.  For each arc $a_i$, let $D_i$ be a small closed disk neighborhood of $a_i$ in $\text{int}(\D)$, and let $D$ be the collection of the disks $D_i$.  In addition, let $U = \bigcup \pd D_i$, an unlink.  Then the link $K \cup U$ bounds the embedded planar surface $P = \D \setminus (\bigcup D)$.  Finally, there is a collection of $r-1$ arcs in $P$ connecting the $r$ components of $\pd D_i$, and thickening these arcs yields a collection $B$ of $r-1$ bands in $P$ connecting components of $D$.  Since $P$ is planar, it follows that $P \setminus (\bigcup B)$ is an embedded annulus with $K$ as one of its boundary components, and so $K$ is also isotopic in $S^3$ to the other boundary component of $P \setminus (\bigcup B)$, as shown in Figure~\ref{fig:abstract}.  We conclude that $(D,B)$ is a disk-band presentation for a ribbon disk $\D'$ such that $\pd \D' = K$, completing the proof.
\end{proof}

\begin{figure}[h!]
  \centering
  \includegraphics[width=.4\linewidth]{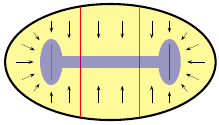}
  \caption{An example of the annulus $P \setminus (\bigcup B)$ from the proof of Lemma~\ref{lem:fusion}.}
\label{fig:abstract}
\end{figure}

\begin{remark}\label{rmk:unknot}
Lemma~\ref{lem:fusion} gives a quick argument that no nontrivial knot has ribbon number one, since $r(K) < 2$ implies $\F(K) < 1$, and $\F(K) = 0$ if and only if $K$ is the unknot.  In addition, the inequality in Lemma~\ref{lem:fusion} is equality for a nontrivial knot $K$ with $r(K) = 2$.
\end{remark}

\begin{remark}
The ribbon disk $\D'$ constructed in Lemma~\ref{lem:fusion} need not be identical to the ribbon disk $\D$ used as input.  A disk-band presentation for a ribbon disk $\D$ with $\F(\D) = 3$ and $r(\D) = 3$ appears in Figure~\ref{fig:fusion}, where $\pd \D$ is the knot $9_{41}$.  Carrying out the process in Lemma~\ref{lem:fusion} yields a new ribbon disk $\D'$, where $\F(\D') = 2$ but $r(\D') = 4$.  See Figure~\ref{fig:fusion}.
\end{remark}

\begin{remark}
The higher-dimensional analogue of this example is discussed in detail in~\cite{YasudaBase}; in the context of 2-knots, ribbon number is replaced with ``crossing number" and fusion number is related to ``base index."  See Section~\ref{sec:doubles} for further details on ribbon 2-knots.
\end{remark}

\begin{figure}[h!]
  \centering
  \includegraphics[width=.5\linewidth]{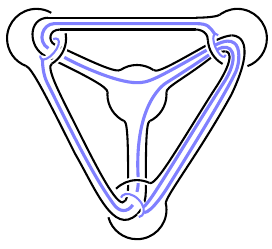}
  \caption{A disk-band presentation for a ribbon disk $\D$ for the knot $9_{41}$ with $\F(\D) = r(\D) = 3$.  Carrying out the procedure described in \ref{lem:fusion} yields another ribbon disk $\D'$ with $\F(\D') = 2$ but $r(\D') = 4$.}
\label{fig:fusion}
\end{figure}

To conclude this section, we can use these lemmas to determine exact values for $r(T_{p,q} \# \overline{T_{p,q}})$.  This is a particularly nice class of ribbon knots; for instance, $\F(T_{p,q} \# \overline{T_{p,q}}) = \min\{p,q\} - 1$ \cite{JMZ}, and their fibered ribbon disks have been studied in great detail (see ~\cite{FHRkS,qdisks}).

\begin{proposition}\label{prop:gensq}
Let $T_{p,q}$ denote the $(p,q)$-torus knot, with $p,q>0$.  Then
\[ r(T_{p,q} \# \overline{T_{p,q}}) = (p-1)(q-1).\]
\end{proposition}

\begin{proof}
The equation above is symmetric in $p$ and $q$, so suppose without loss of generality that $p<q$.  It is well-known that $g(T_{p,q}) = \frac{1}{2}(p-1)(q-1)$ and genus is additive under connected sum, so that $g(T_{p,q} \# \overline{T_{p,q}}) = (p-1)(q-1)$.  It follows from Lemma~\ref{lem:genus} that
\[ r(T_{p,q} \# \overline{T_{p,q}}) \geq (p-1)(q-1).\]
On the other hand, the standard diagram $D$ for $T_{p,q}$ has $c(D) = q(p-1)$ and $\ell(D) = p-1$.  Thus, by Lemma~\ref{lem:underpass},
\[ r(T_{p,q} \# \overline{T_{p,q}}) \leq q(p-1) - (p-1) = (p-1)(q-1),\]
completing the proof.
\end{proof}

\begin{remark}
The statement in Proposition~\ref{prop:gensq} can be compared with Theorem 1.2 from~\cite{YasudaEval}, which gives the ribbon crossing number of a spun torus knot (see Section~\ref{sec:doubles} for definitions).
\end{remark}

\section{Doubles of ribbon disks}~\label{sec:doubles}

The notion of a ribbon knot extends to higher dimensions.  A 2-knot $\K \subset S^4$ is a \emph{ribbon 2-knot} if $\K$ bounds an immersed $D^3 \subset S^4$ with only ribbon intersections.  In this case, the ribbon $D^3$ has a higher-dimensional disk-band presentation consisting of $n$ 3-dimensional 0-handles and $n-1$ 3-dimensional 1-handles that can intersect the 0-handles in some number of 2-disks.  A more detailed and technical description of ribbon 2-knots appears in Section 2.2 of~\cite{CKS}, in which each 0-handle is called a \emph{base} and each 1-handle is called a \emph{band}.

Analogous to the classical case, an intersection of a base and band, which must be a 2-disk by definition, is called a \emph{ribbon intersection}.  The \emph{ribbon crossing number} $\text{r-cr}(\K)$ of a ribbon 2-knot is defined to be the minimum number of ribbon intersections contained in an immersed $D^3$ as described above.  If $\D \subset S^3$ is a ribbon disk for a classical knot $K$ with $r$ ribbon intersections, then as noted in the introduction, we can perturb $\D$ near its ribbon intersections to construct an embedded disk $D^4$, which we will call an \emph{embedded ribbon disk}, and which we will also denote $\D$ in an abuse of notation.  The union of two copies of $\D \subset D^4$ glued along the identity map is a ribbon 2-knot $\K(\D)$ in $S^4$, called the \emph{double of $\D$}, and $\K(\D)$ bounds an immersed $D^3 \subset S^4$ with the same number of ribbon singularities as the immersed disk $\D$.  See, for instance, Figure 2.5 and the surrounding discussion of~\cite{CKS} for further details.  In this case, $K$ is called the \emph{equatorial ribbon knot} of $\K(\D)$, and it is also known that every ribbon 2-knot can arise from such a construction.  In addition, we have the following lemma, a folk theorem that we have not seen elsewhere in print.

\begin{lemma}\label{lem:double}
If $K$ is a ribbon knot bounding a ribbon disk $\D$, then $\Delta_K(t) = f(t) \cdot f(t^{-1})$, where $f(t)=\Delta_{\K(\D)}(t)$ is the Alexander polynomial of the double $\K(\D)$, which is also equal to the Alexander polynomial $\Delta_{\D}(t)$ of the embedded ribbon disk $\D$.
\end{lemma}

\begin{proof}
The statement that $\Delta_K(t)=f(t)\cdot f(t^{-1})$ with $f(t)$ the Alexander polynomial of $\D$ is the content of~\cite[Corollary~15.11]{FNOP} and is explained in detail in Chapter~15 of that same reference.

First, we briefly sketch this argument and then explain why $f(t)$ is in fact also the Alexander polynomial of the double $\K(\D)$.  Let $\Lambda := \Z[t^{\pm 1}]$. We denote the order of a finitely generated $\Lambda$-module $M$ by $\ord(M)\in \Lambda$ (see~\cite[p.50]{Hillman}).
By definition we have
\[\begin{array}{rcl} \Delta_K(t)&=&\ord(H_1(S^3\setminus\nu  K;\Lambda)), \\ \Delta_{\K(\D)}(t)&=&\ord(H_1(S^4\setminus\nu \K(\D);\Lambda)), \\
\Delta_{\D}(t)&=&\ord(H_1(D^4\setminus\nu  \D;\Lambda)).\end{array}\] 
Next we consider the following excerpt of the exact sequence of the pair 
\[  H_2(D^4\setminus \nu \D,S^3\setminus \nu  K;\Lambda)\xrightarrow{\partial_2} H_1(S^3\setminus \nu K;\Lambda)\xrightarrow{i_*} H_1(D^4\setminus \nu \D;\Lambda). \]
Since $\D$ is an embedded ribbon disk, the exterior of $\D$ has a description by $4$-dimensional handle attachments to the exterior of $K$ that only uses $2$- and $3$-handles. Therefore the map $i_*$ 
is an epimorphism. Using Poincar\'e duality, the Universal Coefficient Theorem and the fact that the exterior of $K$ and the boundary of the exterior of $\D$ have isomorphic Alexander module structures, one sees that the left hand map is a monomorphism
and that $\operatorname{ord}(H_2(D^4\setminus \nu \D,S^3\setminus\nu  K;\Lambda))= 
\overline{\ord(H_1(D^4\setminus \nu\D;\Lambda))}=\Delta_{\D}(t^{-1})$. 
Next we recall the following purely algebraic result:  \cite[Lemma~5]{Levine} say that given a short exact sequence $0 \to A\to B\to C\to 0$ of finitely generated $\Lambda$-modules  the order of the middle module is the product of the orders of the outer modules.
It follows from this discussion  that $\Delta_{K}(t)=f(t)\cdot f(t^{-1})$
where $f(t)=\Delta_{\D}(t)$.

We now turn our attention to the main statement, namely that $f(t)$ is in fact equal to the Alexander polynomial of the double $\K(\D)$.  To this end, we express the $4$-sphere $S^4$ as a union of two $4$-balls, $S^4=D_1^4\cup D_2^4$, and we denote by $\D_1\subset D_1^4$ and $\D_2\subset D_2^4$ two copies of the embedded ribbon disk with $\K(\D)=\D_1\cup \D_2$.
This leads to the decomposition $S^4\setminus \nu \K(\D)=(D_1^4\setminus \nu \D_1)\cup_{S^3\setminus \nu K}(D_2^4\setminus \nu \D_2)$. We consider the corresponding Mayer--Vietoris sequence with $\Lambda$-coefficients:
\[ H_1(S^3\setminus \nu K;\Lambda)
\to H_1(D^4_1\setminus \nu \D_1;\Lambda)\oplus
H_1(D^4_2\setminus \nu \D_2;\Lambda)\to
H_1(S^4\setminus \nu \K(D);\Lambda)\to 0.\]
Since the inclusion induced maps 
$ H_1(S^3\setminus \nu K;\Lambda)
\to H_1(D^4_i\setminus \nu \D_i;\Lambda)$
are the same  as the inclusion induced epimorphism
$ H_1(S^3\setminus \nu K;\Lambda)
\to H_1(D^4\setminus \nu \D;\Lambda)$, 
we obtain from the above long exact sequence  that $H_1(D^4\setminus \nu \D;\Lambda)\cong H_1(S^4\setminus \nu \K(\D);\Lambda)$. 

We set $f(t):= \Delta_{\K(\D)}(t)=\ord(H_1(S^4\setminus \nu \K(\D);\Lambda))$.
It follows from the above discussion that 
$\Delta_K(t)=\Delta_{\D}(t)\cdot \Delta_{\D}(t^{-1})=f(t)\cdot f(t^{-1})$. 
\end{proof}

Yasuda proved the next theorem.

\begin{theorem}\label{thm:yasuda}\cite{YasudaAlex}
Suppose $\K$ is a ribbon 2-knot bounding a ribbon $D^3$ with $r$ ribbon intersections.  Then there exists a representative $f(t) = a_0 + a_1t + \dots + a_rt^r$ of the Alexander polynomial of $\K$ such that
\[ |a_i| \leq \binom{r}{i}\]
for $0 \leq i \leq r$.
\end{theorem}

\begin{proof}[Proof of Theorem~\ref{thm:main1}]
Suppose $K$ bounds a ribbon disk $\D$ with $r(K) = r$ ribbon intersections.  Then the double $\K(\D)$ is a ribbon 2-knot bounding a ribbon $D^3$ with $r$ ribbon intersections.  By Theorem~\ref{thm:yasuda}, there exists a representative $f(t)$ of the Alexander polynomial of $\K(\D)$ whose coefficients satisfy the stated inequality, and by Lemma~\ref{lem:double}, we have $\Delta_K(t) = f(t) \cdot f(t^{-1})$, completing the proof.
\end{proof}

\begin{proof}[Proof of Corollary~\ref{cor:main1}]
Suppose $K$ bounds a ribbon disk $\D$ with $r(K) = n$ ribbon intersections.  By Theorem~\ref{thm:main1}, we can express $\Delta_K(t)$ as $f(t) \cdot f(t^{-1})$, where $f(t) = a_0 + a_1t + \dots + a_rt^r$ and $|a_i| \leq \binom{r}{i}$ for all $i$.  Observe that $\det(K) = |\Delta_K(-1)| = |f(-1)|^2$.  Now, we compute
\[ |f(-1)| \leq |a_0| + |a_1| + \dots + |a_r| \leq \binom{r}{0} + \binom{r}{1} + \dots + \binom{r}{r} = 2^r.\]
However, since $\det(K)$ is odd, $|f(-1)|$ is also odd, and so $|f(-1)| \leq 2^r-1$, completing the proof.
\end{proof}

\section{Ribbon codes and Alexander polynomials}\label{sec:codes}

In this section, we develop new machinery to better understand the possible Alexander polynomials of ribbon knots.  To each ribbon disk, we will associate a tree with marked and labeled edges, called a ribbon code.  Formally, a \emph{ribbon code} is a tree $\Gamma$ with $n$ vertices $v_1,\dots,v_n$, such that the union of the interiors of the edges in $\Gamma$ contains a finite number of distinguished points $\mu_1,\dots,\mu_r$, called \emph{markings}, with each marking $\mu_{\ell}$ labeled with an integer in the set $\{\pm 1, \dots ,\pm n\}$.

Recall the definition of a disk-band presentation $(D,B)$ for a ribbon disk $\D$ from Section~\ref{sec:prelim}.  Here we will also assume that a disk-band presentation is oriented, so that each disk in $D$ and each band in $B$ has a positive normal direction that agrees with the positive normal direction for the ribbon disk $\D$.  Suppose $(D,B)$ is an oriented disk-band presentation, with the disks in $D$ labeled $D_1,\dots,D_n$ and the bands in $B$ labeled $B_1,\dots,B_{n-1}$.  Construct a graph $\Gamma$ from $(D,B)$ by associating a vertex $v_i$ to each disk $D_i$ and connecting two vertices $v_{i_1}$ and $v_{i_2}$ with an edge $e_j$ if the corresponding band $B_j$ has its two opposite boundary edges in the disks $D_{i_1}$ and $D_{i_2}$.  Note that the homotopy type of $\Gamma$ is the same as that of $\D$, and so $\Gamma$ is a tree.

As we follow the band $B_j$ from $D_{i_1}$ to $D_{i_2}$, we add markings to the edge $e_j$ corresponding to the ribbon intersections of $B_j$ with the disks in $D$.  If a marking $\mu_{\ell}$ corresponds to a ribbon intersection of $B_j$ with the disk $D_k$, we label the marking $\pm k$, with the sign decided as follows:  Each marking has a \emph{local direction}, an arrow which points toward the component of $\Gamma \setminus \mu_{\ell}$ containing the vertex $v_k$ associated with $D_k$, and this induces a local direction of the band $B_j$ at the ribbon intersection with $D_k$.  If the local direction on $B_j$ agrees with the positive normal orientation of $D_k$ at the ribbon intersection, $\mu_{\ell}$ is labeled $+k$.  Otherwise, the local direction of $B_j$ disagrees with the orientation of $D_k$, and $\mu_{\ell}$ is labeled $-k$.  If $(D,B)$ is a disk-band presentation for a disk $\D$ with $r$ ribbon intersections, and if $\Gamma$ is the corresponding ribbon code, $\Gamma$ contains a total of $r$ markings.

\subsection{Some guiding examples}

In Figure~\ref{fig:ex1}, we see disk-band presentations for the Stevedore knot $6_1$ (at left) and the square knot $3_1 \# \overline{3_1}$ at right, along with their corresponding ribbon codes.  We have included the local direction at the markings for reference, but the reader should note that this information is redundant, since the local directions are determined uniquely by the labelings.

\begin{figure}[h!]
\begin{subfigure}{.48\textwidth}
  \centering
  \includegraphics[width=.9\linewidth]{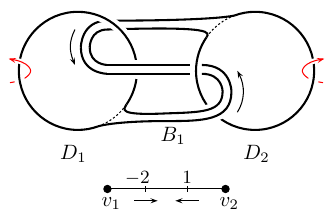}
\end{subfigure}
\begin{subfigure}{.48\textwidth}
  \centering
  \includegraphics[width=.9\linewidth]{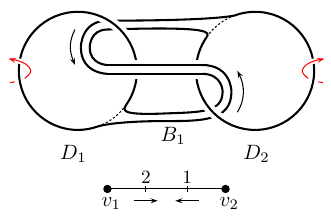}
\end{subfigure}
  \caption{Ribbon codes induced by disk-band presentations for $3_1 \# \overline{3_1}$ (left) and $6_1$ (right).}
\label{fig:ex1}
\end{figure}

In Figure~\ref{fig:ex2}, we see two disk-band presentations for two different knots that induce the same ribbon code.  In particular, the ribbon code is not affected by any homotopy of the bands supported outside of a small neighborhood of the interior of the disks $D$.  While such a homotopy does not change the combinatorics of the ribbon intersections, we are allowed to pass bands through each other, tie a local knot in a band, add full twists to a band, and change the cyclic ordering along which multiple bands attach to a given disk.

\begin{figure}[h!]
\begin{subfigure}{.48\textwidth}
  \centering
  \includegraphics[width=.9\linewidth]{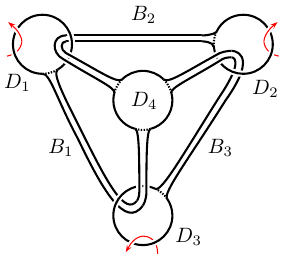}
  \label{fig:braid1}
\end{subfigure}
\begin{subfigure}{.48\textwidth}
  \centering
  \includegraphics[width=.9\linewidth]{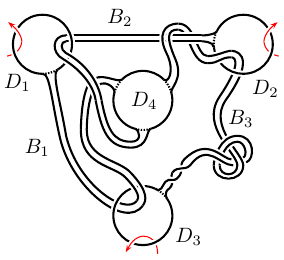}
\end{subfigure}
\begin{subfigure}{.48\textwidth}
  \centering
  \includegraphics[width=.6\linewidth]{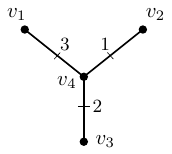}
\end{subfigure}
  \caption{Two disk-band presentations for different knots that induce the same ribbon code   
  }
\label{fig:ex2}
\end{figure}

\begin{remark}
Although adding or deleting a full twist from a band does not change the corresponding ribbon code, adding a half twist can alter the code significantly, due to the requirement that the disks and bands be consistently oriented.  If a band is modified by a half-twist, then the normal direction of one of the disks attached to the band must also be changed, along with the normal direction of the bands attached to that disk, and so on.
\end{remark}

We say that two ribbon codes $\Gamma$ and $\Gamma'$ are \emph{isomorphic} if there is a isomorphism $\varphi:\Gamma \rightarrow \Gamma'$ of the underlying graphs that induces an isomorphism of the markings in the following sense:  Let $\{v_1,\dots,v_n\}$ and $\{v_1',\dots,v_n'\}$ denote the vertex sets of $\Gamma$ and $\Gamma'$, respectively.  Then $\varphi$ induces a permutation $\sigma \in S_n$ by the rule $\sigma(i)$ satisfies $v'_{\sigma(i)} = \varphi(v_i)$.  We also require $\varphi$ induces a bijection between the markings of $\Gamma$ and $\Gamma'$, and if a marking $\mu_{\ell}$ is labeled $\pm i$, then the marking $\varphi(\mu_{\ell})$ is labelled $\pm \sigma(i)$.

The proof of the next lemma is straightforward.

\begin{lemma}\label{lem:finite}
For any fixed $n,r>0$, there are finitely many possible ribbon codes with $n$ vertices and $r$ markings (up to isomorphism).
\end{lemma}

\subsection{Alexander polynomials from ribbon codes}

What is perhaps less obvious is the next lemma, which we will use to show that the Alexander polynomial of a ribbon knot $K$ can be computed directly from its ribbon code.  To this end, suppose that $\Gamma$ is a ribbon code with $r$ markings $\mu_1,\dots,\mu_r$, and for each marking $\mu_{\ell}$, let $\sgn(\mu_{\ell}) = \pm 1$ denote the sign of its label.  In addition, define $\gamma_{\ell} \subset \Gamma$ to be the unique path in $\Gamma$ from the vertex $v_i$ to the marking $\mu_{\ell}$.  Finally, define the $\ell$-th \emph{marking function} $g_{\ell}:\{1,\dots,r\} \setminus \{\ell\} \rightarrow \{-1,0,1\}$ by
\[ g_{\ell}(m) = \begin{cases}
\sgn(\mu_m) & \text{if $\mu_m \in \gamma_{\ell}$ and the local direction at $\mu_m$ agrees with the direction of $\gamma_{\ell}$} \\
- \sgn(\mu_m) & \text{if $\mu_m \in \gamma_{\ell}$, the local direction at $\mu_m$ disagrees with the direction of $\gamma_{\ell}$} \\
0 & \text{if $\mu_m \notin \gamma_{\ell}$.}
\end{cases}\]

\begin{proposition}\label{prop:block}
Suppose that $K$ bounds a ribbon disk $\D$ with $r$ ribbon intersections and ribbon code $\Gamma$.  Then $K$ admits a $2r \times 2r$ block Seifert matrix $A$ of the form
\[ A = \begin{bmatrix}
0 & X \\ Y & Z \end{bmatrix}\]
such that the blocks $X$ and $Y$ are uniquely determined by $\Gamma$.  In particular, the entries $x_{m\ell}$ of $X$ and $y_{m\ell}$ of $Y$ are determined by the following rules:
\begin{enumerate}
\item If $m \neq \ell$, then $x_{m\ell} = y_{\ell m} = g_{\ell}(m)$.
\item If $\sgn(\mu_{\ell}) = 1$, then $x_{\ell \ell} = 0$ and $y_{\ell \ell} = -1$.
\item If $\sgn(\mu_{\ell}) = -1$, then $x_{\ell \ell} = 1$ and $y_{\ell \ell} = 0$.
\end{enumerate}
As a consequence, $X - Y^T = \text{Id}_{\,r}$.
\end{proposition}

\begin{proof}
Suppose $(D,B)$ is a disk-band presentation for a ribbon disk $\D$ bounded by $K$ that gives rise to a ribbon code $\Gamma$, where $\D$ has $r$ ribbon intersections.  Let $F$ be the Seifert surface obtained by smoothing the ribbon intersections of $\D$ as in the proof of Lemma~\ref{lem:genus}.  We will construct a basis for $H_1(F)$ in order to find a Seifert matrix $A$.

Each ribbon intersection in $\D$ corresponds to two arcs $a_{\ell}$ and $a_{\ell}'$ in $\D$, where $a_{\ell}'$ is properly embedded and $a_{\ell}$ is embedded in $\text{int}(\D)$.  Let $N_i$ denote a regular neighborhood of $a_{\ell}$ in $\D$, and let $\A_{\ell}$ denote $\pd N_{\ell}$, oriented counterclockwise with respect to the normal direction of $\D$.  In addition, let $A_{\ell}$ be the annulus in $F$ whose boundary consists of $\A_{\ell}$ and another curve made up of two arcs in $K$ and two arcs in $\text{int}(F)$, as shown in Figure~\ref{fig:ii1}.  By construction, we have
\[ F = \D \setminus \left( \bigcup N_{\ell}  \cup \bigcup a_{\ell}' \right) \cup \bigcup A_{\ell}.\]
In this way, we can view the $\A_{\ell}$ curves in both $\D$ and in $F$, yielding half of the curves in our basis of $H_1(F)$.  Since the disks $N_i$ are pairwise disjoint, $\text{lk}(\A_{\ell},\A_m^+) = 0$ for all $\ell$ and $m$, and so the upper left block of $A$ is the zero block.

As above, we denote the $r$ markings of $\Gamma$ by $\mu_1,\dots,\mu_r$, where $\mu_{\ell}$ is associated to the ribbon intersection between a disk $D_i$ and a band $B_j$ yielding arcs $a_{\ell}$ and $a'_{\ell}$ in $\D$.  Let $b_{\ell}$ be an embedded path in $\D$ from $a_{\ell}$ to $a'_{\ell}$, chosen so that $b_{\ell}$ avoids all other arcs of the form $a_m$ and crosses $a'_m$ at most once for $m \neq {\ell}$.  Then $b_{\ell}$ corresponds to the path $\gamma_{\ell}$ in the tree $\Gamma$ from the vertex $v_i$ (associated to the disk $D_i$) to the marking $\mu_{\ell}$, and the arcs $a'_m$ that $b_{\ell}$ crosses correspond to the markings $\mu_m$ contained in $\gamma_{\ell}$.  Construct $\n_{\ell}$ in $F$ by starting with the arc $b_{\ell} \setminus \left( \bigcup N_{\ell}  \cup \bigcup a'_{\ell} \right)$ and connecting the endpoints in the annulus with an arc in $A_{\ell}$, as in Figures~\ref{fig:ii1} and~\ref{fig:ii2}.  By construction, $|\A_{\ell} \cap \n_m| = \delta_{\ell m}$, and so $\{\A_1,\dots,\A_r,\n_1,\dots,\n_r\}$ is a basis for $H_1(F)$.

\begin{figure}[h!]
  \centering
  \includegraphics[width=0.9\linewidth]{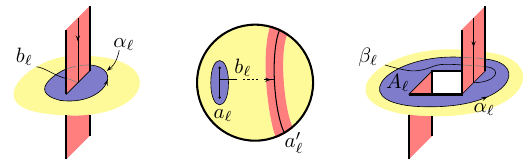}
  \caption{The interaction of $\A_{\ell}$ and $\n_{\ell}$ when $\sgn(\mu_{\ell}) = 1$ 
  }
\label{fig:ii1}
\end{figure}

\begin{figure}[h!]
  \centering
  \includegraphics[width=0.9\linewidth]{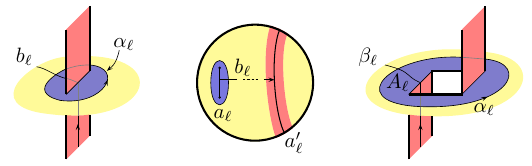}
  \caption{The interaction of $\A_{\ell}$ and $\n_{\ell}$ when $\sgn(\mu_{\ell}) = -1$ 
  }
\label{fig:ii2}
\end{figure}

We can use Figures~\ref{fig:ii1} and~\ref{fig:ii2} to compute $\text{lk}(\A_{\ell},\n_{\ell}^+)$ and $\text{lk}(\n_{\ell},\A_{\ell}^+)$.  The case in which $\sgn(\mu_{\ell}) = 1$ is shown in Figure~\ref{fig:ii1}; we can verify
\[ \text{lk}(\A_{\ell},\n_{\ell}^+) = 0 \quad \text{and} \quad \text{lk}(\n_{\ell},\A_{\ell}^+) = -1.\]
The case in which $\sgn(\mu_{\ell}) = -1$ is shown in Figure~\ref{fig:ii2}; we have
\[ \text{lk}(\A_{\ell},\n_{\ell}^+) = 1 \quad \text{and} \quad \text{lk}(\n_{\ell},\A_{\ell}^+) = 0.\]

For $m \neq \ell$, we have that $\text{lk}(\A_m,\n_{\ell}^+) = \text{lk}(\n_{\ell},\A_m^+)  \neq 0$ if and only if $\n_{\ell}$ passes through the disk $N_m$, as in Figure~\ref{fig:ij}.  These intersections correspond precisely with the markings contained in the path $\gamma_{\ell}$.  Any such marking $\mu_m$ corresponds to the ribbon intersection giving rise to $N_m$ and $\A_m$, and we can see that
\[ \text{lk}(\A_m,\n_{\ell}^+) = \text{lk}(\n_{\ell},\A_m^+) = \pm 1,\]
depending on whether the direction of $b_{\ell}$ and the normal direction at $a'_m$ agree ($+1$) or disagree ($-1$).  Note that $\sgn(\mu_{\ell})$ indicates whether the local direction at $a_m$ agrees or disagrees with the normal direction at $a'_m$.  Thus, $\text{lk}(\A_m,\n_{\ell}^+) = \text{lk}(\n_{\ell},\A_m^+) = 1$ precisely when the direction of $\gamma_{\ell}$ agrees with the local direction at $\mu_{\ell}$ and $\sgn(\mu_{\ell}) = 1$ or when the direction of $\gamma_{\ell}$ disagrees with the local direction at $\mu_{\ell}$ and $\sgn(\mu_{\ell}) = - 1$.  Conversely, $\text{lk}(\A_m,\n_{\ell}^+) = \text{lk}(\n_{\ell},\A_m^+) = -1$ precisely when the direction of $\gamma_{\ell}$ agrees with the local direction at $\mu_{\ell}$ and $\sgn(\mu_{\ell}) = -1$ or when the direction of $\gamma_{\ell}$ disagrees with the local direction at $\mu_{\ell}$ and $\sgn(\mu_{\ell}) = 1$. Succinctly,
\[ \text{lk}(\A_m,\n_{\ell}^+) = \text{lk}(\n_{\ell},\A_m^+) = g_{\ell}(m).\]
Noting that the $(m,\ell)$-th entry of the matrix $X - Y^T$ is given by $\text{lk}(\A_m,\n_{\ell}^+) - \text{lk}(\n_{\ell},\A_m^+)$, we have $X - Y^T = \text{Id}_{r}$, completing the proof.

\begin{figure}[h!]
  \centering
  \includegraphics[width=0.9\linewidth]{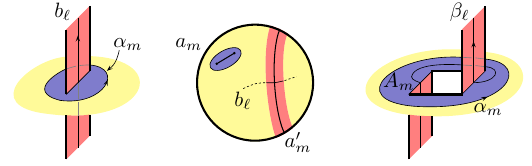}
  \caption{The interaction of $\A_m$ and $\n_{\ell}$ when the direction of $b_{\ell}$ agrees with the normal direction at $a_m$ 
  }
\label{fig:ij}
\end{figure}

\end{proof}

\begin{remark}
The ribbon code does not determine the block $Z$ in Proposition~\ref{prop:block} above.  This block, whose entries are of the form $\text{lk}(\n_m,\n_{\ell}^+)$, is determined by the homotopy type and twisting of the bands $B$ away from a neighborhood of the ribbon intersections.  Indeed, although the $X$ and $Y$ blocks are the same in the Seifert matrices produced by the examples shown in Figure~\ref{fig:ex2}, the $Z$ blocks are quite different.
\end{remark}

As a corollary, we obtain

\begin{corollary}\label{cor:codes}
If $K$ and $K'$ are two ribbon knots with ribbon disks giving rise to isomorphic ribbon codes, then $\Delta_K = \Delta_{K'}$.
\end{corollary}

\begin{proof}
Since $K$ and $K'$ induce the same ribbon code, it follows from Proposition~\ref{prop:block} that $K$ and $K'$ admit Seifert matrices $A$ and $A'$, respectively, of the form
\[ A = \begin{bmatrix}
0 & X \\ Y & Z \end{bmatrix} \quad \text{and} \quad
A' = \begin{bmatrix}
0 & X' \\ Y' & Z' \end{bmatrix}.\]
We compute
\[ \Delta_K(t) = \det(A - tA^T) = \begin{vmatrix} 0 & X - tY^T \\ Y - tX^T & Z \end{vmatrix} .\]
Since $0$ commutes with $Y - tX^T$, it follows from elementary linear algebra that
\[ \begin{vmatrix} 0 & X - tY^T \\ Y - tX^T & Z \end{vmatrix} = \det( 0 \cdot Z - (X-tY^T) \cdot (Y-tX^T)) = \det((X-tY^T)(Y-tX^T)).\]
An identical calculation shows that
\[ \Delta_{K'}(t) = \det((X-tY^T)(Y-tX^T)).\]
\end{proof}

Another nice application of Proposition~\ref{prop:block} helps us understand knots with similar but non-identical ribbon codes.  Although this is not relevant to our later analysis, it may be of independent interest for the future study of ribbon codes.

\begin{corollary}
Suppose ribbon knots $K$ and $K'$ bound disks with ribbon codes that are identical except for the sign of a single marking.  Then $\det(K) = \det(K')$.
\end{corollary}

\begin{proof}
Suppose that $K$ bounds a disk giving rise to $\Gamma$ and $K'$ bounds a disk giving rise to $\Gamma'$, where $\Gamma$ and $\Gamma'$ are identical except for corresponding markings $\mu_m$ and $\mu_m'$, which are labeled $+i$ in $\Gamma$ and $-i$ in $\Gamma'$ for some $i \geq 1$.  Then $K$ and $K'$ admit block Seifert matrices $A$ and $A'$, respectively, of the form
\[ A = \begin{bmatrix}
0 & X \\ Y & Z \end{bmatrix} \quad \text{and} \quad
A' = \begin{bmatrix}
0 & X \\ Y & Z' \end{bmatrix}\]
by Proposition~\ref{prop:block}, where $Y^T = X - \text{Id}_r $ and $(Y')^T = X' - \text{Id}_r$.  It follows that
\[ \det(K) = \Delta_K(-1) =  \det((X+Y^T)(Y+X^T)) = (\det(X+Y^T))^2.\]
Similarly, $\det(K') = (\det(X'+(Y')^T))^2$.  We prove the corollary by comparing the values of $x_{ij} + y_{ji}$ and $x'_{ij} + y'_{ji}$.

Let $g_{\ell}$ and $g'_{\ell}$ denote the marking functions corresponding to $\Gamma$ and $\Gamma'$, respectively.  Now, observe that the only entries of the matrices $X + Y^T$ and $X' + (Y')^T$ that depend on $\sgn(\mu_m)$ and $\sgn(\mu_m')$ are the entries $x_{\ell m} + y_{m \ell}$ and $x'_{\ell m} + y'_{m \ell}$, so that $X + Y^T$ and $X'+(Y')^T$ are identical outside of their $m$th columns.  Moreover, the local directions at $\mu_m$ and $\mu'_m$ are independent of their sign, and so for any $\ell \neq m$, there are three possibilities:
\begin{enumerate}
\item $g_{\ell}(m) = \sgn(\mu_m)$ and $g'_{\ell}(m) = \sgn(\mu'_m) = -\sgn(\mu_m)$,
\item $g_{\ell}(m) = -\sgn(\mu_m)$ and $g'_{\ell}(m) = -\sgn(\mu'_m) = \sgn(\mu_m)$, or
\item $g_{\ell}(m) = g'_{\ell}(m) = 0$.
\end{enumerate}
In each of the three cases, we have $x'_{\ell m } = y'_{m \ell} = -x_{\ell m} = -y_{\ell m}$, and thus $x'_{\ell m} + y'_{m \ell} = -(x_{\ell m} + y_{m \ell})$ whenever $\ell \neq m$.  On the other hand, if $\sgn(\mu_m) = 1$, then $x_{\ell \ell} = 0$ and $y_{\ell \ell} = -1$, and $\sgn(\mu'_m) = -1$, so that $x'_{\ell \ell} = 1$ and $y'_{\ell \ell} = 0$.  Otherwise, $\sgn(\mu_m) = -1$, so that $x_{\ell \ell} = 1$ and $y_{\ell \ell} = 0$, and $\sgn(\mu'_m) = 1$, so that $x'_{\ell \ell } = 0$ and $y'_{\ell\ell} = -1$.  In either case, we again get $x'_{\ell \ell} + y'_{\ell \ell} = -(x_{\ell\ell}+ y_{\ell \ell})$.  We conclude that the matrices $X + Y^T$ and $X'+(Y')^T$ are identical except that the $m$th column of $X'+(Y')^T$ has entries opposite the $m$th column of $X + Y^T$, and therefore $\det(X'+ (Y')^T) = - \det(X + Y^T)$, completing the proof.
\end{proof}

\begin{remark}
Eisermann noted in Remark 5.21 of~\cite{eisermann} so-called \emph{band crossing changes} do not have any effect on the Alexander polynomial, using an argument similar to the one here.  Other authors have also examined and discussed these ideas.  In~\cite{KSTI}, the authors also analyzed the Seifert surfaces and matrices arising from ribbon disks so that they could obstruct knots from being \emph{simple-ribbon knots} using Alexander polynomials.  Their analysis is similar to that appearing in the proof of Proposition~\ref{prop:block}.  In~\cite{bai}, the author used ``ribbon diagrams" and ``ribbon graphs" to compute Alexander polynomials.  The work in that paper is also similar (but not identical) to the proof of Proposition~\ref{prop:block}.
\end{remark}

\section{Classifying low-complexity ribbon codes and Alexander polynomials}\label{sec:code}

In this section, we use ribbon codes to classify possible Alexander polynomials for ribbon knots $K$ with small ribbon numbers, which we will use in Section~\ref{sec:tab} to tabulate ribbon numbers for knots with 11 or fewer crossings.

\subsection{Simplification of ribbon codes}

In order to determine all possible low-complexity ribbon codes, we will use several tools to reduce certain configurations to simpler ones.  We define the \emph{fusion number} $\F(\Gamma)$ of a ribbon code $\Gamma$ to be the number of edges in the graph $\Gamma$ and the \emph{ribbon number} $r(\Gamma)$ to be the number of markings (noting that these quantities coincide with the corresponding complexities of the ribbon disks they represent).  Additionally, by Corollary~\ref{cor:codes}, a ribbon code uniquely determines an Alexander polynomial, and so we let $\Delta_{\Gamma}(t)$ or simply $\Delta_\Gamma$ denote the Alexander polynomial determined by the ribbon code $\Gamma$.  For simplicity, we also use $\Delta_K$ instead of $\Delta_K(t)$ in this section.  The next five lemmas provide parameters which will allow us to narrow our search for possible Alexander polynomials.

\begin{lemma}\label{lem:simp1a}
Suppose $\Gamma$ is a ribbon code such that an edge of $\Gamma$ contains no markings.  Then there is a ribbon code $\Gamma'$ such that $\Delta_{\Gamma'} = \Delta_{\Gamma}$, $\F(\Gamma') < \F(\Gamma)$, and $r(\Gamma') = r(\Gamma)$.
\end{lemma}

\begin{proof}
Let $(D,B)$ be a disk-band presentation for a ribbon disk $\D$ giving rise to the ribbon code $\Gamma$, where $D = \{D_1,\dots,D_n\}$ and $B = \{B_1,\dots,B_{n-1}\}$.  Possibly after reindexing, suppose that the band $B_{n-1}$ corresponds to the edge $e$ in $\Gamma$ with no markings, and $B_{n-1}$ is attached to disks $D_{n-1}$ and $D_n$.  Since $e$ contains no markings, the interior of $B_{n-1}$ is disjoint from $D$.  Let $D_{n-1}' = D_{n-1} \cup B_{n-1} \cup D_n$.  Then $D' = \{D_1,\dots,D_{n-2},D'_{n-1}\}$ is a collection of embedded disks.  If in addition we let $B' = \{B_1,\dots,B_{n-2}\}$, we have that $(D',B')$ is a disk-band presentation for a ribbon disk $\D'$ such that $\pd \D' = \pd \D$, $\F(\D') < \F(\D)$, and $r(\D') = r(\D')$.  We conclude that the ribbon code $\Gamma'$ for $\D'$ satisfies $\Delta_{\Gamma'}(t) = \Delta_{\Gamma}(t)$ and $\F(\Gamma') < \F(\Gamma)$, as desired.
\end{proof}

\begin{lemma}\label{lem:simp1b}
Suppose $\Gamma$ is a ribbon code such that an edge of $\Gamma$ contains consecutive markings with labels $i$ and $-i$.  Then there is a ribbon code $\Gamma'$ such that $\Delta_{\Gamma'} = \Delta_{\Gamma}$ and $r(\Gamma') < r(\Gamma)$.
\end{lemma}

\begin{proof}
Suppose $e_j$ is an edge of $\Gamma$ with consecutive markings labeled $i$ and $-i$.  We can use $\Gamma$ to construct a disk-band presentation $(D,B)$ for a ribbon disk $\D$ such that $\D$ gives rise to $\Gamma$, and such that the band $B_j$ corresponding to $e_j$ passes through the disk $D_i$ in opposite directions as shown at left in Figure~\ref{fig:simp}.  In particular, we arrange the construction so that there is an embedded disk $E$ whose boundary is the endpoint union of an arc in $B_j$ and an arc in $D_i$.  We can use the disk $E$ to build an isotopy of the band $B_j$ that removes the two opposite ribbon intersections, yielding a ribbon disk $\D'$ such that $\pd \D'$ is isotopic to $\pd \D$ and $r(\D') < r(\D)$.  Thus, the ribbon code $\Gamma'$ for $\D'$ satisfies $\Delta_{\Gamma'} = \Delta_{\Gamma}$ and $r(\Gamma') < r(\Gamma)$, as desired.
\end{proof}

\begin{figure}[h!]
  \centering
  \includegraphics[width=0.8\linewidth]{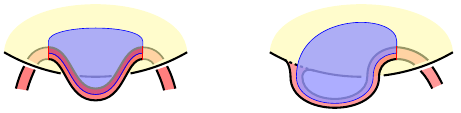}
  \caption{The local pictures described in Lemmas~\ref{lem:simp1b} and~\ref{lem:simp2}}
\label{fig:simp}
\end{figure}

\begin{lemma}\label{lem:simp2}
Suppose $\Gamma$ is a ribbon code such that an edge incident to a vertex $v_i$ has $\pm i$ as its marking closest to $v_i$.   Then there is a ribbon code $\Gamma'$ such that $\Delta_{\Gamma'} = \Delta_{\Gamma}$ and $r(\Gamma') < r(\Gamma)$.
\end{lemma}

\begin{proof}
Suppose $e_j$ is an edge of $\Gamma$ incident to a vertex $v_i$ and the marking in $e_j$ closest to $v_i$ is $\pm i$.  We can use $\Gamma$ to construct a disk-band presentation $(D,B)$ for a ribbon disk $\D$ such that $\D$ gives rise to $\Gamma$, and such that the band $B_j$ corresponding to $e_j$ is attached the disk $D_i$ corresponding to the vertex $v_i$ immediately passes through $D_i$ as shown at right in Figure~\ref{fig:simp}.  In particular, we arrange the construction so that there is an embedded disk $E$ whose boundary is the endpoint union of an arc in $B_j$ and arc in $D_i$ and such that $E \cap \D = \pd E$.  We can use the disk $E$ to build an isotopy of the band $B_j$ that removes the ribbon intersection, yielding a ribbon disk $\D'$ such that $\pd \D'$ is isotopic to $\pd \D$ and $r(\D') < r(\D)$.  As above, the corresponding ribbon code $\Gamma'$ for $\D'$ satisfies $\Delta_{\Gamma'} = \Delta_{\Gamma}$ and $r(\Gamma') < r(\Gamma)$.
\end{proof}

\begin{remark}
Note that Lemmas~\ref{lem:simp1b} and~\ref{lem:simp2} imply that there exist disks $\mathcal{D}$ associated to ribbon codes satisfying the hypotheses of the lemmas such that $\mathcal{D}$ can be simplified, but this is not necessarily the case \emph{for all} disks corresponding to these ribbon codes.  Indeed, the Kinoshita-Terasaka knot $11n_{42}$ shown in Figure~\ref{fig:KT} (adapted from Figure 2.7 of~\cite{CKS}) bounds a disk $\D$ satisfying the hypotheses of Lemma~\ref{lem:simp2}; nevertheless $r(11n_{42}) = 3$ by~\cite{mizuma} and so the ribbon number of this particular $\D$ cannot be reduced, even though the corresponding ribbon code can be simplified.
\end{remark}

\begin{figure}[h!]
  \centering
  \includegraphics[width=0.6\linewidth]{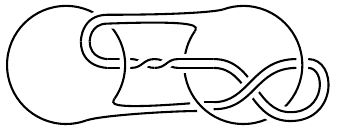}
  \caption{A disk-band presentation for the Kinoshita-Terasaka knot $11n_{42}$}
\label{fig:KT}
\end{figure}

\begin{lemma}\label{lem:simp3}
Suppose $\Gamma$ is a ribbon code containing a vertex $v_i$ of valence one such that $\Gamma$ contains no marking labeled $\pm i$.  Then there is a ribbon code $\Gamma'$ such that $\Delta_{\Gamma'}(t) = \Delta_{\Gamma}(t)$, $\F(\Gamma') < \F(\Gamma)$, and $r(\Gamma') \leq r(\Gamma)$.
\end{lemma}

\begin{proof}
After relabeling if necessary, suppose that $e_{n-1}$ is the only edge incident to $v_n$ in $\Gamma$, and that $\Gamma$ contains no marking labeled $\pm n$.  As above, we can use $\Gamma$ to construct a disk-band presentation $(D,B)$ for a ribbon disk $\D$ giving rise to $\Gamma$, where $D = \{D_1,\dots,D_n\}$, $B=\{B_1,\dots,B_{n-1}\}$, $D_n$ corresponds to $v_n$, and $B_{n-1}$ corresponds to $e_{n-1}$.  Let $D' = \{D_1,\dots,D_{n-1}\}$ and $B' = \{B_1,\dots,B_{n-2}\}$.  Then $(D',B')$ is a disk-band presentation for a ribbon disk $\D'$.  In addition, since $\Gamma$ contains no marking labeled $\pm n$, it follows that no band in $B$ meets the disk $D_n$ in its interior.  Thus, $B_{n-1} \cup D_n$ is an embedded disk which is disjoint from the bands in $B'$, and we can construct an isotopy of $\pd \D$ across this disk which pushes one arc of $\pd(B_{n-1} \cup D_n)$ to the arc of $\pd (\bigcup D')$ along which the other end of $B_{n-1}$ is attached. It follows that $\pd \D'$ and $\pd \D$ are isotopic knots.

In addition, $\D'$ has $k \geq 0$ fewer ribbon intersections than $\D$, where $k$ is the number of markings contained on the edge $e_{n-1}$.  Hence the ribbon code $\Gamma'$ for $\D'$ satisfies $\Delta_{\Gamma'} = \Delta_{\Gamma}$, $\F(\Gamma') < \F(\Gamma)$, and $r(\Gamma') \leq r(\Gamma)$, as desired.
\end{proof}

\begin{lemma}\label{lem:simp4}
Suppose $\Gamma$ is a ribbon code containing a vertex $v_i$ of valence two such that $\Gamma$ contains no marking labeled $\pm i$.  Then there is a ribbon code $\Gamma'$ such that $\Delta_{\Gamma'} = \Delta_{\Gamma}$, $\F(\Gamma') < \F(\Gamma)$, and $r(\Gamma') = r(\Gamma)$.
\end{lemma}

\begin{proof}
After relabeling if necessary, suppose $e_{n-2}$ and $e_{n-1}$ are the only edges incident to the vertex $v_n$, and no edge in $\Gamma$ contains a marking labeled $\pm n$.  Use $\Gamma$ to construct a disk-band presentation $(D,B)$ for a ribbon disk $\D$ giving rise to $\Gamma$, where $D = \{D_1,\dots,D_n\}$ and $B=\{B_1,\dots,B_{n-1}\}$, with indices corresponding as above.  Since the marking $\pm n$ does not appear in $\Gamma$, none of the bands in $B$ meet $D_n$ in its interior.  Thus, if we let $B'_{n-2} = B_{n-2} \cup D_n \cup B_{n-1}$, we have that $B'_{n-2}$ is disjoint from the other bands in $B$ and shares two boundary arcs with arcs in $\pd (\bigcup D)$, so we can view $B'_{n-2}$ as a new band.  Letting $D' = \{D_1,\dots,D_{n-1}\}$ and $B' = \{B_1,\dots,B_{n-3},B'_{n-2}\}$, we have that $(D',B')$ is another disk-band presentation for the same ribbon disk $\D$.  Therefore, the corresponding ribbon code $\Gamma'$ satisfies $\Delta_{\Gamma'} = \Delta_{\Gamma}$, $\F(\Gamma') < \F(\Gamma)$, and $r(\Gamma') = r(\Gamma)$.
\end{proof}

\begin{remark}
The conclusions of Lemmas~\ref{lem:simp3} and \ref{lem:simp4} do \emph{not} hold for vertices of $\Gamma$ of valence three or greater.  For example, the ribbon code $\Gamma$ shown in Figure~\ref{fig:ex2} has no marking labeled $\pm 4$ corresponding to the vertex $v_4$ of valence three, but $\Gamma$ admits no obvious simplification.  This is not relevant to what follows, but we note that in this case, we can make a new ribbon code $\Gamma'$ such that $\F(\Gamma') < \F(\Gamma)$ but such that $r(\Gamma') > r(\Gamma)$, similar to the process shown in Figure~\ref{fig:fusion}.
\end{remark}

The next lemma also helps to cut down the number of cases we consider in our analysis.

\begin{lemma}\label{lem:simp5}
Suppose $\Gamma$ and $\Gamma'$ are ribbon codes with isomorphic underlying graphs but such that every label of $\Gamma'$ is opposite that of $\Gamma$.  Then $\Delta_{\Gamma} = \Delta_{\Gamma'}$.
\end{lemma}

\begin{proof}
In this case, we can use $\Gamma$ and $\Gamma'$ to construct ribbon disks $\D$ and $\D'$, respectively, such that $\pd \D' = \overline{\pd \D}$.  As mirror images have identical Alexander polynomials, the statement follows immediately.
\end{proof}

Guided by the hypotheses of the above lemmas, we call a ribbon code $\Gamma$ \emph{irreducible} if

\begin{enumerate}
\item Every edge of $\Gamma$ contains at least one marking,
\item No edge contains consecutive markings labeled $i$ and $-i$,
\item No marking closest to a vertex has the same label as that vertex, and
\item Every vertex of valence one or two appears at least once as the label of some marking.
\end{enumerate}

Otherwise $\Gamma$ is called \emph{reducible}.

\begin{proposition}\label{prop:irred}
For any reducible ribbon code $\Gamma$, there exists an irreducible ribbon code~$\Gamma'$ such that $r(\Gamma') \leq r(\Gamma)$ and $\Delta_{\Gamma'} = \Delta_{\Gamma}$.
\end{proposition}

\begin{proof}
If $\Gamma$ is reducible, we may apply Lemma~\ref{lem:simp1a}, \ref{lem:simp1b}, \ref{lem:simp2}, \ref{lem:simp3}, or \ref{lem:simp4} to produce the desired~$\Gamma'$.
\end{proof}

We have one final lemma that will aid in our search.

\begin{lemma}\label{lem:size}
If $\Gamma$ is an irreducible ribbon code, then
 \[\F(\Gamma) \leq r(\Gamma).\]
 \end{lemma}
 
 \begin{proof}
 If $\F(\Gamma) > r(\Gamma)$, then $\Gamma$ contains more edges than markings, so at least one edge does not contain a marking, and $\Gamma$ is reducible.
\end{proof}

\subsection{Enumerating possible Alexander polynomials for $r=2$ and $r=3$}

For the remainder, we set the convention that (unless otherwise specified) our ribbon disks are oriented with the normal direction pointing out of the paper, as in Figures~\ref{fig:ex1} and~\ref{fig:ex2}, and so we omit the red normal vector in the figures below.  For any nonnegative integer $r$, define $\R_r$ to be the set of all possible Alexander polynomials of knots $K$ such that $r(K) \leq r$.  By definition, we have $\R_0 \subset \R_1 \subset \R_2 \subset \dots$, and by Theorem~\ref{thm:main1}, we know $\R_r$ is finite.  Additionally, $\R_0 =\R_1= \{1\}$ (see Remark~\ref{rmk:unknot}).  

\begin{proposition}\label{prop:r2}
$\R_2 = \{1, \Delta_{3_1 \# \overline{3_1}}, \Delta_{6_1}\}$.
\end{proposition}

\begin{proof}
Suppose that $\Gamma$ satisfies $r(\Gamma) = 2$.  By Proposition~\ref{prop:irred}, we can suppose without loss of generality that $\Gamma$ is irreducible, in which case Lemma~\ref{lem:size} implies that $\F(\Gamma) \leq 2$.  If $\F(\Gamma) = 2$, then $\Gamma$ has three vertices of valence one or two but only two markings, so that at least one of these vertices does not appear as a label and $\Gamma$ is reducible.  It follows that $\F(\Gamma) = 1$, so that $\Gamma$ has a two vertices $v_1$ and $v_2$ and a single edge $e_1$ from $v_1$ to $v_2$.  Additionally, $e_1$ contains two marking $\mu_1$ and $\mu_2$, in order.  Since $\Gamma$ is irreducible, it follows that the corresponding markings satisfy $x_1 = \pm 2$ and $x_2 = \pm 1$.  By Lemma~\ref{lem:simp5}, we may suppose without loss of generality that $x_2 = 1$, leaving us with two possible ribbon codes $\Gamma_1$ (in which $x_1 = 2$) and $\Gamma_2$ (in which $x_1 = -2)$.  At right in Figure~\ref{fig:ex1}, we can see the ribbon code $\Gamma_1$, which gives rise to the knot $6_1$, and at left in Figure~\ref{fig:ex1}, we see $\Gamma_2$, yielding $3_1 \# \overline{3_1}$.
\end{proof}


\begin{proposition}\label{prop:r3}
$\R_3 =  \{1, \Delta_{3_1 \# \overline{3_1}}, \Delta_{6_1}, \Delta_{8_8}, \Delta_{8_9}, \Delta_{9_{27}}, \Delta_{9_{41}}, \Delta_{10_{137}}, \Delta_{10_{153}}, \Delta_{11n_{116}}\}$.
\end{proposition}

\begin{proof}
Suppose $\Gamma$ satisfies $r(\Gamma)=3$.  By Proposition~\ref{prop:irred}, we can suppose without loss of generality that $\Gamma$ is irreducible, in which case Lemma~\ref{lem:size} implies that $\F(\Gamma) \leq 3$, so that $\Gamma$ has at most four vertices.  We consider the cases $\F(\Gamma) = 1$, $\F(\Gamma) =2$, and $\F(\Gamma) = 3$ separately.

First, suppose $\F(\Gamma) = 1$, so that $\Gamma$ has two vertices $v_1$ and $v_2$, a single edge $e_1$ from $v_1$ to $v_2$, and three markings $\mu_1,\mu_2,\mu_3$ in order with labels $x_1,x_2,x_3$, respectively.  Since $\Gamma$ is irreducible, we have $x_1 = \pm 2$ and $x_3 = \pm 1$, and by Lemma~\ref{lem:simp5}, we may suppose that $x_3 = 1$.  Again using irreducibility, the possible markings are $(x_1,x_2,x_3) = (\pm 2, 1,1)$ or $(x_1,x_2,x_3) = (\pm 2, \pm 2, 1)$, but using a graph isomorphism, the marking $(x_1,x_2,x_3) = (-2,-2,1)$ yields the same ribbon code as the marking $(-2, 1, 1)$.  Thus, we need only consider the cases $(x_1,x_2,x_3) = (\pm 2, 1, 1)$.  In Figure~\ref{fig:f1}, we see two ribbon disks which yield these two codes, and their boundaries are $11n_{116}$ (with $r=0$) and $10_{153}$, respectively.

\begin{figure}[h!]
  \centering
  \includegraphics[width=0.8\linewidth]{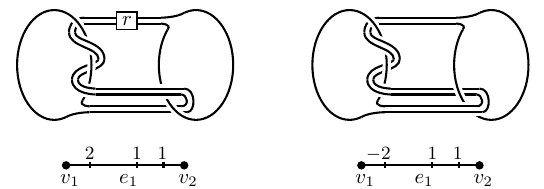}
  \caption{Disk-band presentations and corresponding ribbon codes for $\D$ with $\pd \D = 11n_{116}$ (left, with $r=0$) and $10_{153}$ (right).}
\label{fig:f1}
\end{figure}

Next, suppose $\F(\Gamma)=2$, so that $\Gamma$ has three vertices $v_1$, $v_2$, and $v_3$ and two edges $e_1$ (from $v_1$ to $v_2$) and $e_2$ (from $v_2$ to $v_3$).  After applying an isomorphism if necessary, we may assume that $e_1$ contains two markings $\mu_1$ and $\mu_2$ (in order) and $e_2$ contains a single marking $\mu_3$.  Moreover, since $\Gamma$ is irreducible, the corresponding labels must satisfy $(x_1,x_2,x_3) = (\pm 2, \pm 3, \pm 1)$, since each marking is the closest marking to one or two of the vertices, and each label must appear exactly once.  By Lemma~\ref{lem:simp5}, we may assume that $x_3 = 1$, and so the four possible labelings are $(x_1,x_2,x_3) = (\pm 2, \pm 3, 1)$.  As shown in Figure~\ref{fig:f2}, these ribbon codes are induced by ribbon disks whose boundaries are the knots $8_9$ (with $r=0$), $8_8$, $10_{137}$, and $10_{129}$, where $\Delta_{8_8} = \Delta_{10_{129}}$.

\begin{figure}[h!]
  \centering
  \includegraphics[width=0.8\linewidth]{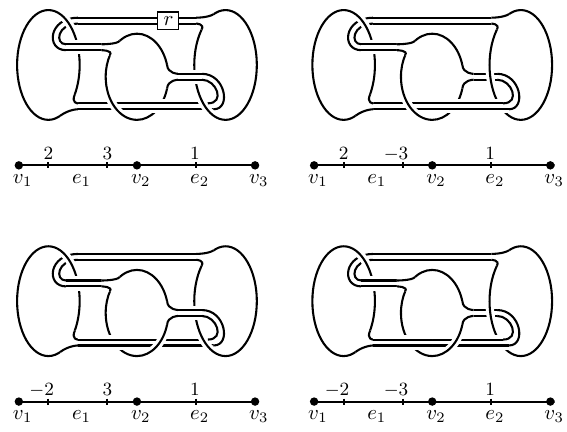}
  \caption{Disk-band presentations and corresponding ribbon codes for $\D$ with $\pd \D = 8_9$ (top left, with $r=0$), $8_8$ (top right), $10_{137}$ (bottom left), and $10_{129}$ (bottom right).}
\label{fig:f2}
\end{figure}

Finally, suppose $\F(\Gamma) = 3$, so that $\Gamma$ has four vertices $v_1,v_2,v_3$, and $v_4$.  Since some vertex, say $v_4$, does not appear as a marking, it must be the case that the valence of $v_4$ is at least three.  It follows that the valence of $v_4$ is exactly three, and the other three vertices have valence one.  Let $e_1$, $e_2$, and $e_3$ be the edges of $\Gamma$, where $e_i$ connects $v_i$ and $v_4$.  Each edge $e_i$ must contain one marking $\mu_i$ labeled $x_i$, where the label $x_i$ is not equal to $\pm i$ by irreducibility.  Up to isomorphism, the only possible labelings are $(x_1,x_2,x_3) = (\pm 3, \pm 1, \pm 2)$.  Up to isomorphism and using Lemma~\ref{lem:simp5}, there are only two possibilities:  Either all signs agree, so $(x_1,x_2,x_3) = (3,1,2)$, or one sign differs from the other two, so $(x_1,x_2,x_3) = (3,1,-2)$.  The two cases are shown at left and right in Figure~\ref{fig:f3}, in which the ribbon disks have boundary $9_{41}$ (with $r=0$) and $9_{27}$ (with $r=s=t=0$), respectively.  This exhausts all possibilities for $\Gamma$, completing the proof.

\begin{figure}[h!]
  \centering
  \includegraphics[width=0.8\linewidth]{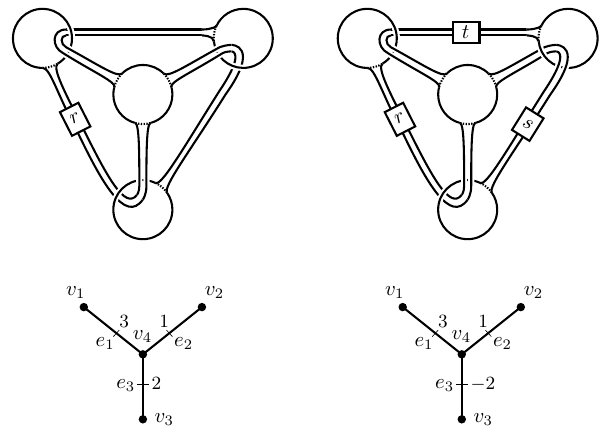}
  \caption{Disk-band presentations and corresponding ribbon codes for $\D$ with $\pd \D = 9_{41}$ (left, with $r=0$) and $9_{27}$ (right, with $r=s=t=0$). 
  }\label{fig:f3}
\end{figure}

\end{proof}

\section{Tabulating ribbon numbers up to 11 crossings}\label{sec:tab}

In this section, we tabulate the ribbon numbers for ribbon knots in the knot table up to 11 crossings.  Tables~\ref{table1} and~\ref{table2} include the knot name, Alexander polynomial, genus, determinant, ribbon number, and justification for the upper and lower bounds for the ribbon numbers.  Except for the ribbon numbers, the data in the tables was retrieved from the KnotInfo database~\cite{knotinfo}.  For succinctness, we list $\Delta_K$ as a tuple of coefficients, and we omit coefficients implied by the symmetry of $\Delta_k$.  For example, $\Delta_{6_1}(t) = 2t^{-1} - 5 + 2t$ is written as $(2,-5)$ in the table.

One additional justification we need comes from a proposition of Mizuma and Tsutsumi and uses the \emph{crosscap number} $\gamma(K)$ of $K$, the minimum nonorientable genus of a nonorientable surface bounded by $K$.  

\begin{proposition}\cite{MizTsu}\label{prop:MT}
If $K$ is a ribbon knot such that $r(K)=2$, then either $g(K) = 1$ or $\gamma(K) \leq 2$.
\end{proposition}

For the upper bounds, we typically use a known disk-band presentation or a known symmetric union presentation in conjunction with Lemma~\ref{lem:underpass}.  For instance, in the template at left in Figure~\ref{fig:f1}, the box labeled $r$ represents some number of full twists on two strands, and when $r=-1$, the corresponding knot is $11n_{49}$. In the template at top left in Figure~\ref{fig:f2}, the box labeled $r$ also represents full twists and when $r=1$, the corresponding knot is $11n_{37}$.  In the templates in Figure~\ref{fig:f3}, boxes labeled $r$, $s$, and $t$ also represent full twists.  When $r=-1$ in the template at left, the corresponding knot is $11n_{83}$.  In the template at right, when $(r,s,t) = (-1,0,0)$, $(0,-1,0)$, or $(0,0,1)$, the corresponding knots are $11n_{21}$, $11n_4$, or $11n_{172}$, respectively.

In Tables~\ref{table1} and~\ref{table2}, the justification~\cite{lamm1} refers to the symmetric union presentations depicted in Figure 16 of that paper.  The justification~\cite{lamm2} refers to Figure 3 (for the knot $10_{87}$) and Figure 4 (for the other knots in the table).  The justification~\cite{lamm3} refers to Table 2 (for the knots $11n_{67}$, $11n_{73}$, $11n_{74}$, and $11n_{97}$), Table 3 (for the knots $10_{99}$, $11a_{58}$, $11a_{103}$, $11a_{165}$, $11a_{201}$, $3_1 \# 8_{10}$, $3_1 \# 8_{11}$), or the Appendix (for the other knots in the table).  For the knots for which we have not determined the exact ribbon number, we give multiple possibilities.

For the lower bound on the ribbon number of $K \in  \{11n_{42},11n_{67},11n_{74},11n_{97}\}$, we used KnotInfo~\cite{knotinfo} to verify that $\gamma(K) > 2$, and so Proposition~\ref{prop:MT} can be applied (note, however, that we cited~\cite{mizuma} for $11n_{42}$, since historically that result appeared before Proposition~\ref{prop:MT}).  The upper bounds for $10_{42}$, $10_{75}$, $10_{87}$, and $11n_{39}$ use the ribbon disks shown in Figure~\ref{fig:extra}, and the upper bounds for $11a_{28}$, $11a_{35}$, $11a_{36}$, $11a_{87}$, $11a_{96}$, $11a_{115}$, $11a_{169}$, and $11a_{316}$ are shown in Figure~\ref{fig:extra2}.

\begin{table}[ht]
\centering
\caption{Ribbon number data and justifications for knots up to 10 crossings}
\label{table1}
\begin{tabular}{ccccccc}
\hline
$K$ & $\Delta_K$ & $\det(K)$ & $g(K)$ & $r(K)$ & lower & upper  \\
\hline \hline
$0_1$ & (1) & 1 & 0 & 0 & & \\
\hline
$6_1$ & $(2,-5)$ & 9 & 1 & 2 & Rmk.~\ref{rmk:unknot} & Fig.~\ref{fig:f1} \\
\hline
$3_1 \# \overline{3_1}$ & $(1,-2,3)$ & 9 & 2 & 2 & Prop.~\ref{prop:gensq} & Prop.~\ref{prop:gensq} \\
\hline
$8_8$ & $(2,-6,9)$ & 25 & 2 & 3 & Prop.~\ref{prop:r2} & Fig.~\ref{fig:f2} \\
\hline
$8_9$ & $(1,-3,5,-7)$ & 25 & 2 & 3 & Prop.~\ref{prop:r2} & Fig.~\ref{fig:f2} \\
\hline
$8_{20}$ & $(1,-2,3)$ & 9 & 2 & 2 & Rmk.~\ref{rmk:unknot} & \cite{lamm1} \\
\hline
$4_1 \# \overline{4_1}$ & $(1,-6,11)$ & 25 & 2 & 3 & Prop.~\ref{prop:r2} & Lem.~\ref{lem:underpass} \\
\hline
$9_{27}$ & $(1,-5,11,-15)$ & 49 & 3 & 3 & Lem.~\ref{lem:genus} & Fig.~\ref{fig:f3} \\
\hline
$9_{41}$ & $(3,-12,19)$ & 49 & 2 & 3 & Prop.~\ref{prop:r2} & Fig.~\ref{fig:ex2} \\
\hline
$9_{46}$ & $(2,-5)$ & 9 & 1 & 2 & Rmk.~\ref{rmk:unknot} & \cite{lamm1} \\
\hline
$10_{3}$ & $(6,-13)$ & 25 & 1 & 4 & Prop.~\ref{prop:r3} & \cite{lamm1} \\
\hline
$10_{22}$ & $(2,-6,10,-13)$ & 49 & 3 & 4 & Prop.~\ref{prop:r3} & \cite{lamm1} \\
\hline
$10_{35}$ & $(2,-12,21)$ & 49 & 2 & 4 & Prop.~\ref{prop:r3} & \cite{lamm1} \\
\hline
$10_{42}$ & $(1,-7,19,-27)$ & 81 & 3 & 4 & Prop.~\ref{prop:r3} & Fig.~\ref{fig:extra} \\
\hline
$10_{48}$ & $(1,-3,6,-9,11)$ & 49 & 4 & 4 & Lem.~\ref{lem:genus} & \cite{lamm1} \\
\hline
$10_{75}$ & $(1,-7,19,-27)$ & 81& 3 & 4 & Prop.~\ref{prop:r3} & Fig.~\ref{fig:extra} \\
\hline
$10_{87}$ & $(2,-9,18,-23)$ & 81& 3 & 4 & Prop.~\ref{prop:r3} & Fig.~\ref{fig:extra} \\
\hline
$10_{99}$ & $(1,-4,10,-16,19)$ & 81& 4 & 4 & Lem.~\ref{lem:genus} & \cite{lamm3} \\
\hline
$10_{123}$ & $(1,-6,15,-24,29)$ & 121& 4 & 4, 5 & Lem.~\ref{lem:genus} & \cite{lamm1} \\
\hline
$10_{129}$ & $(2,-6,9)$ & 25 & 2 & 3 & Prop.~\ref{prop:r2} & Fig.~\ref{fig:f2} \\
\hline
$10_{137}$ & $(1,-6,11)$ & 25 & 2 & 3 & Prop.~\ref{prop:r2} & Fig.~\ref{fig:f2} \\
\hline
$10_{140}$ & $(1,-2,3)$ & 9 & 2 & 2 & Rmk.~\ref{rmk:unknot} & \cite{lamm1} \\
\hline
$10_{153}$ & $(1, -1, 1, -3)$ & 1 & 3 & 3 & Lem.~\ref{lem:genus} & Fig.~\ref{fig:f1} \\
\hline
$10_{155}$ & $(1, -3, 5, -7)$ & 25 & 3 & 3 & Lem.~\ref{lem:genus} & \cite{lamm1} \\
\hline
$5_1 \# \overline{5_1}$ & $(1,-2,3,-4,5)$ & 25 & 4 & 4 & Prop.~\ref{prop:gensq} & Prop.~\ref{prop:gensq} \\
\hline
$5_2 \# \overline{5_2}$ & $(4,-12,17)$ & 49 & 2 & 4 & Prop.~\ref{prop:r3} & Lem.~\ref{lem:underpass} \\
\hline
\end{tabular}
\end{table}

\begin{table}[ht]
\centering
\caption{Ribbon number data and justifications for knots with 11 crossings}
\label{table2}
\begin{tabular}{ccccccc}
\hline
$K$ & $\Delta_K$ & $\det(K)$ & $g(K)$ & $r(K)$ & lower & upper  \\
\hline \hline
$11a_{28}$ & $(1,-6,15,-24,29)$ & 121 & 4 & 4 & Lem.~\ref{lem:genus} & Fig.~\ref{fig:extra2} \\
\hline
$11a_{35}$ & $(1,-5,14,-25,31)$ & 121 & 4 & 4 & Lem.~\ref{lem:genus} & Fig.~\ref{fig:extra2} \\
\hline
$11a_{36}$ & $(2,-12,28,-37)$ & 121 & 3 & 4 & Prop.~\ref{prop:r3} & Fig.~\ref{fig:extra2} \\
\hline
$11a_{58}$ & $(2,-9,18,-23)$ & 81 & 3 & 4 & Prop.~\ref{prop:r3} & \cite{lamm3} \\
\hline
$11a_{87}$ & $(2, -11, 28, -39)$ & 121 & 3 & 4 & Prop.~\ref{prop:r3} & Fig.~\ref{fig:extra2}  \\
\hline
$11a_{96}$ & $(1, -9, 29, -43)$ & 121 & 3 & 4 & Prop.~\ref{prop:r3} &Fig.~\ref{fig:extra2} \\
\hline
$11a_{103}$ & $(4, -20, 33)$ & 81 & 2 & 4 & Prop.~\ref{prop:r3} & \cite{lamm3} \\
\hline
$11a_{115}$ & $(3, -13, 27, -35)$ & 121 & 3 & 4 & Prop.~\ref{prop:r3} & Fig.~\ref{fig:extra2}  \\
\hline
$11a_{164}$ & $(1, -7, 20, -35, 43)$ & 169 & 4 & 4, 5, 6 & Lem.~\ref{lem:genus} & \cite{lamm3} \\
\hline
$11a_{165}$ & $(2, -9, 18, -23)$ & 81 & 3 & 4 & Prop.~\ref{prop:r3}  & \cite{lamm3} \\
\hline
$11a_{169}$ & $(2, -12, 28, -37)$ & 121 & 3 & 4 & Prop.~\ref{prop:r3} & Fig.~\ref{fig:extra2} \\
\hline
$11a_{201}$ & $(4, -20, 33)$ & 81 & 2 & 4 & Prop.~\ref{prop:r3} & \cite{lamm3} \\
\hline
$11a_{316}$ & $(1, -5, 14, -25, 31)$ & 121 & 4 & 4 & Lem.~\ref{lem:genus} & Fig.~\ref{fig:extra2} \\
\hline
$11a_{326}$ & $(1, -6, 19, -36, 45)$ & 169 & 4 & 4, 5, 6 & Lem.~\ref{lem:genus} & \cite{lamm3} \\
\hline
$11n_{4}$ & $(1, -5, 11, -15)$ & 49 & 3 & 3 & Lem.~\ref{lem:genus} & Fig.~\ref{fig:f3} \\
\hline
$11n_{21}$ & $(1, -5, 11, -15)$ & 49 & 3 & 3 & Lem.~\ref{lem:genus} & Fig.~\ref{fig:f3} \\
\hline
$11n_{37}$ & $(1, -3, 5, -7)$ & 25 & 3 & 3 & Lem.~\ref{lem:genus} & Fig.~\ref{fig:f2} \\
\hline
$11n_{39}$ & $(2, -6, 9)$ & 25 & 2 & 3 & Prop.~\ref{prop:r2} & Fig.~\ref{fig:extra} \\
\hline
$11n_{42}$ & $(1)$ & 1 & 2 & 3 & \cite{mizuma} & \cite{mizuma} \\
\hline
$11n_{49}$ & $(1,0,-3)$ & 1 & 2 & 3 & Prop.~\ref{prop:r2} & Fig~\ref{fig:f1} \\
\hline
$11n_{50}$ & $(2,-6,9)$ & 25 & 2 & 3 & Prop.~\ref{prop:r2} & \cite{lamm3} \\
\hline
$11n_{67}$ & $(2,-5)$ & 9 & 2 & 3 & Prop.~\ref{prop:MT} & \cite{lamm3} \\
\hline
$11n_{73}$ & $(1,-2,3)$ & 9 & 3 & 3 & Lem.~\ref{lem:genus} & \cite{lamm3} \\
\hline
$11n_{74}$ & $(1,-2,3)$ & 9 & 2 & 3 & Prop.~\ref{prop:MT} & \cite{lamm3} \\
\hline
$11n_{83}$ & $(3,-12,19)$ & 49 & 2 & 3 & Prop.~\ref{prop:r2} & Fig.~\ref{fig:f3} \\
\hline
$11n_{97}$ & $(2,-5)$ & 9 & 2 & 3 & Prop.~\ref{prop:MT} & \cite{lamm3} \\
\hline
$11n_{116}$ & $(1,0,-3)$ & 1 & 2 & 3 & Prop.~\ref{prop:r2} & Fig.~\ref{fig:f1} \\
\hline
$11n_{132}$ & $(2,-6,9)$ & 25 & 2 & 3 & Prop.~\ref{prop:r2} & \cite{lamm3} \\
\hline
$11n_{139}$ & $(2,-5)$ & 9 & 1 & 2 & Rmk.~\ref{rmk:unknot} & \cite{lamm3} \\
\hline
$11n_{172}$ & $(1,-5,11,-15)$ & 49 & 3 & 3 & Lem.~\ref{lem:genus} & Fig.~\ref{fig:f3}  \\
\hline
$3_1 \# 8_{10}$ & $(1,-4,10,-16,19)$ & 81 & 4 & 4 & Lem.~\ref{lem:genus} & \cite{lamm3} \\
\hline
$3_1 \# 8_{11}$ & $(2,-9,18,-23)$ & 81 & 3 & 4 & Prop.~\ref{prop:r3} & \cite{lamm3} \\
\hline
\end{tabular}
\end{table}

\begin{figure}[h!]
 \begin{subfigure}{.48\textwidth}
  \centering
  \includegraphics[width=.83\linewidth]{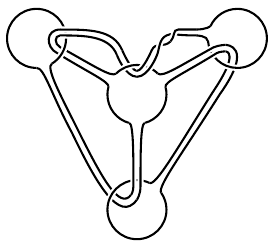}
  \caption{$10_{42}$}
  \label{fig:braid1}
\end{subfigure}
 \begin{subfigure}{.48\textwidth}
  \centering
\includegraphics[width=0.83\linewidth]{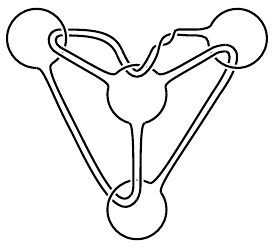}
  \caption{$10_{75}$}
  \label{fig:braid1}
\end{subfigure}
 \begin{subfigure}{.48\textwidth}
  \centering
    \includegraphics[width=0.83\linewidth]{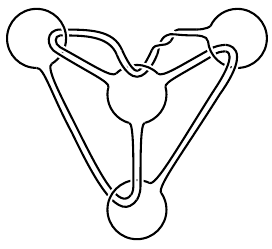}
  \caption{$10_{87}$}
  \label{fig:braid1}
\end{subfigure}
 \begin{subfigure}{.48\textwidth}
  \centering
    \includegraphics[width=0.83\linewidth]{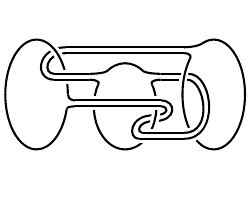}
  \caption{$11n_{39}$}
  \label{fig:braid1}
\end{subfigure}
  \caption{Ribbon disks realizing minimal ribbon numbers, where (A), (B), and (C) are adapted from Figure 3 of~\cite{KSTI}; see Remark~\ref{rmk:KSTI}.}
\label{fig:extra}
\end{figure}

\begin{remark}\label{rmk:KSTI}
The three ribbon disks in Figure~\ref{fig:extra} are adapted from Figure 3 of~\cite{KSTI}; however, we note that figure has an error in the diagrams for $10_{42}$ and $10_{75}$.  The corrected versions were graciously shared with us by the authors via an email exchange, and we used these to construct the presentations in Figure~\ref{fig:extra}.
\end{remark}

\begin{figure}[h!]
  \centering
   \begin{subfigure}{.48\textwidth}
  \centering
    \includegraphics[width=0.83\linewidth]{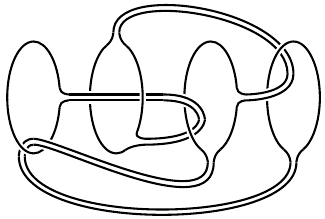} \qquad
  \caption{$11a_{28}$}
  \label{fig:braid1}
\end{subfigure}
 \begin{subfigure}{.48\textwidth}
  \centering
\includegraphics[width=0.83\linewidth]{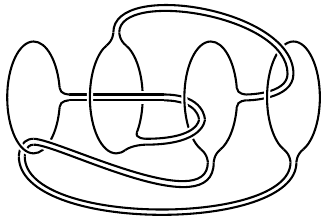} \\
  \caption{$11a_{35}$}
  \label{fig:braid1}
\end{subfigure}
 \begin{subfigure}{.48\textwidth}
  \centering
    \includegraphics[width=0.83\linewidth]{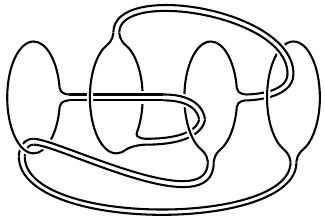} \qquad
  \caption{$11a_{36}$}
  \label{fig:braid1}
\end{subfigure}
 \begin{subfigure}{.48\textwidth}
  \centering
\includegraphics[width=0.83\linewidth]{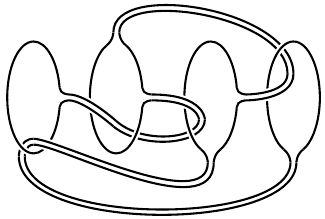} \\
  \caption{$11a_{87}$}
  \label{fig:braid1}
\end{subfigure}
 \begin{subfigure}{.48\textwidth}
  \centering
    \includegraphics[width=0.83\linewidth]{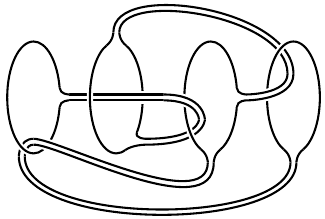} \qquad
  \caption{$11a_{96}$}
  \label{fig:braid1}
\end{subfigure}
 \begin{subfigure}{.48\textwidth}
  \centering
\includegraphics[width=0.83\linewidth]{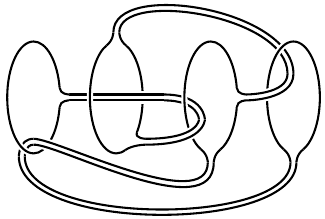} \\
  \caption{$11a_{115}$}
  \label{fig:braid1}
\end{subfigure}
 \begin{subfigure}{.48\textwidth}
  \centering
    \includegraphics[width=0.83\linewidth]{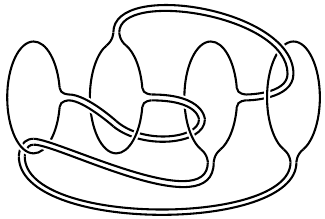} \qquad
  \caption{$11a_{169}$}
  \label{fig:braid1}
\end{subfigure}
 \begin{subfigure}{.48\textwidth}
  \centering
\includegraphics[width=0.83\linewidth]{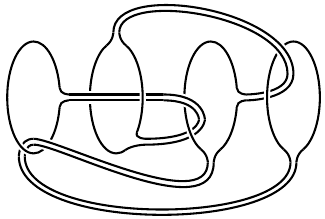}
  \caption{$11a_{316}$}
  \label{fig:braid1}
\end{subfigure}
  \caption{Ribbon disks realizing minimal ribbon numbers}
\label{fig:extra2}
\end{figure}

\section{Jones determinants and ribbon numbers}\label{sec:jones}

A natural question arising from the work above is the following.

\begin{question}\label{q:q1}
To what extent do these results hold for ribbon links?
\end{question}

In the penultimate section, we prove an extension of Corollary~\ref{cor:main1} to links, stated above as Theorem~\ref{thm:main3}, which replaces the knot determinant with a link invariant called the \emph{Jones determinant} introduced by Eisermann in~\cite{eisermann}.

A \emph{ribbon surface} for a link $L \subset S^3$ is a possibly disconnected, possibly nonorientable, immersed surface $\Ss \subset S^3$ such that
\begin{enumerate}
\item $\pd \Ss = L$,
\item $\Ss$ has only ribbon self-intersections, and
\item $\Ss$ has no closed components.
\end{enumerate}
As with a ribbon disk $\D$, we define the \emph{ribbon number} $r(\Ss)$ to be the total number of ribbon intersections contained in $\Ss$.  An \emph{$n$-component ribbon link} $L$ bounds a collection $\D$ of $n$ ribbon disks (which can intersect each other in ribbon singularities).  The \emph{ribbon number} $r(L)$ of $L$ is the minimum of $r(\D)$ taken over all sets of ribbon disks $\D$ bounded by~$L$.

For a diagram $D$ corresponding to a link $L$, we let $\langle D \rangle$ denote the well-known \emph{Kauffman bracket polynomial} in the variable $A$, in which case the Jones polynomial $V_L(q)$ satisfies
\begin{equation}\label{eq:jones}
V_L(q) = (-A)^{-3w(D)} \langle D \rangle
\end{equation}
where $w(D)$ is the \emph{writhe} of $D$, and using the substitution $q = -A^{-2}$ (alternatively, $V_L(t)$ uses the substitution $t = q^2$ or $t = A^{-4}$).  Evaluating $V_L$ at $q=i$ (or $A = e^{-i\pi /4}$), we have
\[ \det(L) = |V_L(i)| = |(e^{\pi i/4})^{-3w(D)} \langle D \rangle_{A = e^{-i\pi/4}}| = |\langle D \rangle_{A = e^{-i\pi/4}}|.\]
Eisermann proved the following.

\begin{theorem}\label{thm:eis}\cite{eisermann}
Suppose $L$ bounds a ribbon surface $\Ss$ such that $\chi(\Ss) = n>0$.  Then $V_L(q)$ is divisible by the Jones polynomial of the $n$-component unlink, $(q+q^{-1})^{n-1}$.
\end{theorem}

By combining the theorem with Equation~\ref{eq:jones}, we get the following corollary.

\begin{corollary}\label{cor:bracket}
Suppose $L$ bounds a ribbon surface $\Ss$ such that $\chi(\Ss) = n>0$, and let $D$ be a diagram for $L$.  Then $\langle D \rangle$ is divisible by $(-A^{-2} - A^2)^{n-1}$.
\end{corollary}

As another consequence of the theorem, motivated by expression for the determinant above, Eisermann defined the \emph{Jones determinant} $\det_n(V_L)$ of a link $L$ bounding a ribbon surface $\Ss$ with $\chi(\Ss) = n$,
\[ \text{det}_n(V_L) = \left(\frac{V_L(q)}{(q+q^{-1})^{n-1}} \right)_{q=i}.\]
When $K$ is a ribbon knot, we have $\det(K) = |\!\det_1(V_K)|$.  
As a part of the proof of Theorem~\ref{thm:eis}, Eisermann computed the difference of the Kauffman bracket for ribbon links related by the local move shown in Figure~\ref{fig:change}.  He proved

\begin{lemma}\label{lem:skein}\cite{eisermann}
Suppose $L$ is a link bounding a ribbon surface of Euler characteristic $n$ with ribbon number $k > 0$.  Then there exists a diagram $D$ for $L$, diagrams $D^n_0$, $D^n_1$, $D^n_2$, $D^n_3$, and $D^n_4$ for links $L^n_i$ bounding ribbon surfaces of Euler characteristic $n$ and ribbon number $k-1$ and diagrams $D^{n+1}_1$ and $D^{n+1}_2$ for links $L^{n+1}_i$ bounding ribbon surfaces of Euler characteristic $n+1$ such that
\begin{eqnarray*}
\langle D \rangle - \langle D^n_0 \rangle &=&
(A^2  - A^{-2}) \left( \langle D^{n+1}_1 \rangle - \langle D^{n+1}_2 \rangle \right) + (A^4 - 1) \left( \langle D^n_1 \rangle  - \langle D^n_2 \rangle \right) \\
&& + \, (A^{-4} - 1) \left( \langle D^n_3 \rangle - \langle D^n_4 \rangle \right).
\end{eqnarray*}
\end{lemma}

\begin{proof}
If $L$ bounds a ribbon surface of Euler characteristic $n$ and ribbon number $k$, there is a diagram of $L$ locally identical to the left frame of Figure~\ref{fig:change}.  The equation above then agrees with Equation (7) from~\cite{eisermann}.
\end{proof}

Lemma~\ref{lem:skein} is a key ingredient in the proof of the next theorem, which specializes to Theorem~\ref{thm:main3} when $L$ is a ribbon link.

\begin{theorem}\label{thm:det}
Suppose $L$ is a link bounding a ribbon surface of Euler characteristic $n >0$ with ribbon number $k \geq 0$.  Then
\[ |\text{det}_n(V_L)| \leq 9^k.\]
\end{theorem}

\begin{proof}
Let $\omega = e^{-i \pi/4}$.  Using the definition of the Jones determinant along with Equation~\ref{eq:jones} and the substitution $q = -A^{-2}$, we have that for any diagram $D$ for $L$,
\begin{equation}\label{eq:det}
|\text{det}_n(V_L)| = \left| \frac{V_L(q)}{(q+q^{-1})^{n-1}}_{q=i} \right| = \left| \frac{(-A)^{-3w(D)} \langle D \rangle}{(-A^{-2} - A^2)^{n-1}}_{A = \omega}\right| = \left| \frac{ \langle D \rangle}{(-A^{-2} - A^2)^{n-1}}_{A = \omega}\right|.
\end{equation}
We induct on the ribbon number $k$ of the surface, call it $\Ss$.  For the base case, suppose that $k=0$.  Since $\chi(\Ss) = n$, it follows that $\Ss$ has at least $n$ disk components (since the maximal Euler characteristic of each component of $\Ss$ is one).  By assumption, $k=0$; hence, $\Ss$ is embedded and $L$ is the split union of an $n$-component unlink $U$ and another (possibly empty) link $L'$.  Let $D = D' \sqcup D''$ be a split diagram in which $D'$ is a crossingless diagram for $U$.  If $L'$ is the empty link, then
\[ |\text{det}_n(V_L)| = \left| \frac{ \langle D' \rangle}{(-A^{-2} - A^2)^{n-1}}_{A = \omega}\right| =\left| \frac{(-A^{-2} - A^2)^{n-1}}{(-A^{-2} - A^2)^{n-1}}_{A = \omega}\right| =1.\]
If $L'$ is not the empty link, then
\[ |\text{det}_n(V_L)| = \left| \frac{ \langle D \rangle}{(-A^{-2} - A^2)^{n-1}}_{A = \omega}\right| =\left| \frac{(-A^{-2} - A^2)^{n}\langle D'' \rangle}{(-A^{-2} - A^2)^{n-1}}_{A = \omega}\right| =0.\]
In either case, $|\det_n(V_L)| \leq 9^0$.

Now, suppose that the inequality holds for any link bounding a ribbon surface of Euler characteristic $n>0$ and with ribbon number $k-1$.  Then there are links and diagrams labeled and yielding the skein identity as in Lemma~\ref{lem:skein}.  Corollary~\ref{cor:bracket} implies that $(-A^{-2}-A^2)^n$ divides $\langle D^{n+1}_i \rangle$, which means that $-A^{-2}-A^2$ is a factor of $\frac{ \langle D^{n+1}_i \rangle}{(-A^{-2} - A^2)^{n-1}}$.  It follows that
\[ \left| \frac{ \langle D^{n+1}_i \rangle}{(-A^{-2} - A^2)^{n-1}}_{A = \omega}\right| = 0.\]
If we rewrite the equation from Lemma~\ref{lem:skein} by dividing by $(-A^{-2} - A^2)^{n-1}$ and evaluating at $A =\omega$, we get
\begin{eqnarray*}
\frac{\langle D \rangle}{(-A^{-2} - A^2)^{n-1}}_{A=\omega} &=& \frac{\langle D^n_0 \rangle}{(-A^{-2} - A^2)^{n-1}}_{A=\omega} - \frac{2\langle D^n_1 \rangle}{(-A^{-2} - A^2)^{n-1}}_{A=\omega} \\
&&+ \,\,\frac{2\langle D^n_2 \rangle}{(-A^{-2} - A^2)^{n-1}}_{A=\omega} - \frac{2\langle D^n_3 \rangle}{(-A^{-2} - A^2)^{n-1}}_{A=\omega} \\
&&+ \,\,\frac{2\langle D^n_4 \rangle}{(-A^{-2} - A^2)^{n-1}}_{A=\omega}.
\end{eqnarray*}
Taking absolute values, using Equation~\ref{eq:det}, and applying the triangle inequality yields
\[ |\text{det}_n(V_L)| \leq |\text{det}_n(V_{L^n_0})| + 2|\text{det}_n(V_{L^n_1})| + 2|\text{det}_n(V_{L^n_2})| + 2|\text{det}_n(V_{L^n_3})| + 2|\text{det}_n(V_{L^n_4})|.\]
Finally, we note that the corresponding links $L^n_i$ bound ribbon surfaces with Euler characteristic $n$ and $k-1$ ribbon intersections, and as such we can use our inductive hypothesis to conclude
\[ |\text{det}_n(V_L)| \leq 9^{k-1} + 2 \cdot 9^{k-1} + 2 \cdot 9^{k-1}+2 \cdot 9^{k-1}+2 \cdot 9^{k-1} = 9^k.\]
\end{proof}

\section{Questions and conjectures}\label{sec:conj}

In this section, we include several interesting avenues of investigation for future research.  The first involves the behavior of ribbon knots under connected sums.  If $K_1$ and $K_2$ are ribbon knots, we can take the connected sum of disks that minimize $r(K_i)$ to show that $r(K_1 \# K_2)$ is at most $r(K_1) + r(K_2)$, but it is unknown whether equality holds in general.

\begin{conjecture}
If $K_1$ and $K_2$ are ribbon knots, then
\[ r(K_1 \# K_2) = r(K_1) + r(K_2).\]
\end{conjecture}

We note that by Lemma~\ref{lem:genus} and the additivity of genus under connected sum, the conjecture holds for any ribbon knots $K_1$ and $K_2$ that satisfy $g(K_i) = r(K_i)$.  For instance, it holds for the knots $T_{p,q} \# \overline{T_{p,q}}$ discussed in Proposition~\ref{prop:gensq}.

We can also consider for which knots the inequality in Lemma~\ref{lem:underpass} is an equality.  We conjecture

\begin{conjecture}
If $K$ is a non-ribbon alternating knot, then
\[ r(K \# \overline{K}) = c(K) - 1.\]
\end{conjecture}

Note that this conjecture is not true for ribbon knots, since $r(6_1) = 2$ and the above argument implies that $r(6_1 \# \overline{6_1}) \leq 4$.  Based on our data, however, we know that the conjecture holds for the knots. $3_1$, $4_1$, $5_1$, $5_2$, and any torus knot of the form $T_{p,2}$.

\begin{question}
Given the usefulness of Alexander polynomials in our computations, can more sophisticated tools such as knot Floer homology give effective lower bounds on ribbon numbers?
\end{question}

We can also consider further investigation of ribbon numbers of links:

\begin{question}
How sharp is the bound $|\det_n(V_L)| \leq 9^{r(L)}$ for ribbon links given by Theorem~\ref{thm:main3}?
\end{question}

For ribbon knots, Corollary~\ref{cor:main1} states a stronger result, which appears to be sharp for low ribbon numbers based on our data.  Can the exponent of 9 in Theorem~\ref{thm:main3} be replaced with something smaller?

\begin{question}
Is there a prescriptive formula for computing (some elements of) $\R_r$ for larger values of $r$?
\end{question}

Forthcoming work~\cite{polymath} computes $\R_4$ (by exhaustion, as in Propositions~\ref{prop:r2} and~\ref{prop:r3} above) and applies the computation to find ribbon numbers for many 12-crossing ribbon knots.  With more data, patterns in the elements of $\R_r$ may begin to emerge.

\bibliographystyle{amsalpha}
\bibliography{ribbon}

\end{document}